\documentclass{article}
\usepackage[utf8]{inputenc}
\usepackage{multicol}
\usepackage{caption}
\usepackage{subcaption}
\usepackage{tikz}
\usepackage{float}
\usepackage{lineno}
\usepackage{graphicx}
\usepackage{comment}
\usepackage[shortlabels]{enumitem}
\usepackage{amsthm,amsmath, amsfonts}
\usepackage{fancyhdr} 
\usepackage{appendix} 
\usepackage{caption} 
\usepackage{hyperref}
\usetikzlibrary{calc}

\newtheorem{theorem}{Theorem}
\newtheorem*{theorem*}{Theorem}
\newtheorem{proposition}[theorem]{Proposition}
\newtheorem{corollary}[theorem]{Corollary}
\newtheorem{lemma}{Lemma}
\newtheorem{definition}{Definition}

\newtheorem{question}{Question}
\newtheorem{example}{Example}
\numberwithin{theorem}{section}
\numberwithin{definition}{section}

\DeclareMathOperator{\ch}{ch}
\DeclareMathOperator{\CH}{CH}

\PassOptionsToPackage{hyphens}{url}\usepackage{hyperref}

\title{Convex geometries representable by at most 5 circles on the plane}
\author{Polly Mathews Jr.\thanks{This work was done as part of \emph{PolyMath REU-2020} program. Team mentor: Kira Adaricheva (Hofstra University). 
Graduate assistants: Madina Bolat (alumna UIUC), Gent Gjonbalaj (Tufts University). 
Team members: Brandon Amerine (Western Oregon University), J. Alexandria Behne (Texas A\&M University), 
Evan Daisy (Amherst College), 
Fernanda Yepez-Lopez (University of Virginia), 
Alexander Frederiksen (University of Central Florida),
Ayush Garg (IIT Delhi), 
Zachary King (University of Central Oklahoma), 
Grace Ma (University of Notre Dame), 
Michelle Olson (California State University, Fullerton), Rohit Pai (NYU), Junewoo Park (Brown University), Cat Raanes (University of British Columbia), Joseph Rogge (UIUC), Raviv Sarch (University of Michigan, Ann Arbor), Sean Riedel (University of California, Santa Cruz), Stephanie Zhou (Rutgers University), James Thompson (UNC-Chapel Hill).}} 

\begin{document}
\maketitle
\begin{abstract} A convex geometry is a closure system satisfying the anti-exchange property. In this work we document all convex geometries on 4- and 5-element base sets with respect to their representation by circles on the plane. All 34 non-isomorphic geometries on a 4-element set can be represented by circles, and of 672 known geometries on a 5-element set,
 we made representations of 623. Of the 49 remaining geometries on a 5-element set, one was already shown not to be representable due to the Weak Carousel property, as articulated by Adaricheva and Bolat (Discrete Mathematics, 2019). In this paper we show that 7 more of these convex geometries cannot be represented by circles on the plane, due to what we term the \emph{Triangle Property}. 
\end{abstract}

\tableofcontents 

\section{Introduction}

In this project we address the problem raised in Adaricheva and Bolat \cite{AdBo19}: whether all geometries 
with convex dimension at most 5 are representable by circles on the plane using the closure operator of convex hull for circles.

An early survey on the topic of convex geometries is given by Edelman and Jamison \cite{EdJa85}, and the theory of infinite convex geometries is developed in Adaricheva, Gorbunov, Tumanov \cite{AGT03}. A recent survey on the topic, which includes both finite and infinite convex geometries,  is in a chapter of \textit{Lattice Theory: Special Topics and Applications} (2016), by Adaricheva and Nation \cite{AN16}.

The representation of finite convex geometries by circles is introduced by G. Cz\'edli in \cite{C14}, where the usual convex hull closure operator acting on sets of points is generalized to the convex hull operator acting on circles. Moreover, he proves that every finite convex geometry of convex dimension 2 can be represented by circles on the plane.

The natural question of whether a finite convex geometry of any convex dimension can be represented by circles on the plane is answered in \cite{AdBo19} \emph{in the negative}: it turns out that the convex hull operator acting on circles in the plane satisfies a condition called the $2\times 3$-\emph{Weak carousel property} - and not all finite convex geometries satisfy it. An example of a geometry with a convex dimension of 6 that does not satisfy this condition is found on a 5-element set.

This left the question of whether any convex geometry on an at most 5-element set, with convex dimension at most 5, could be represented by circles on the plane. 

The goal of the project was to document all existing convex geometries on 4- and 5-element sets, to represent them by circles (when possible), and to describe the features of ``impossible" geometries.

After generating a list of geometries on 4- and 5-element sets, we found representations of all 34 discovered non-isomorphic geometries on 4-element sets and all 623 discovered geometries on 5-element sets. However, 49 of the 672 generated geometries on a 5-element set resisted representation by circles. 

Geometry G4 from the list of 672 non-isomorphic geometries on a 5-element set is the ``impossible" geometry shown not to be representable in \cite{AdBo19}. In this paper we prove that 7 more geometries from the list of 672 are not representable. The proof is based on the Triangle Property (Lemma \ref{triangle} in section \ref{math}).

Three geometries from this new set of ``impossible" geometries have a convex dimension of 4, and one has a convex dimension of 5. {\bf This provides a solution to Problem 1 in \cite{AdBo19} since geometries not representable by circles exist with those convex dimensions.} 

On the other hand, all geometries on 4- and 5-element sets with convex dimension 3 are representable by circles. Thus, the following question remains:

\begin{question}
Does there exist a convex geometry of convex dimension 3 not representable by circles on the plane? 
\end{question}

The paper is organized as follows. We provide all necessary definitions related to convex geometries in section \ref{preliminaries}. In particular the parameter of \emph{convex dimension} is explained in detail. Section \ref{code} provides the background for the algorithm generating all non-isomorphic geometries (anti-matroids) on $n$-element set, for values of $n$ up to $7$. The algorithm was developed around 2013, in the framework of enumerating anti-matroids in the On-line Encyclopedia of Integer Sequences (OEIS).
The existing code was enhanced to additionally compute the \emph{implicational basis} of each geometry and its convex dimension.
Sections \ref{conquer} and \ref{chunks} are devoted to methods for representing geometries by circles. Some basic principles of verifying representations are discussed in section \ref{check}. Sections \ref{Geo} and \ref{Google} provide an overview of different software available for enhancing this area of mathematical work, in our case a GeoGebra and Google Colab toolkit. Finally, section \ref{math} gives the proof of the Triangle Property, and its consequences for a subset of ``impossible" geometries not representable by circles.

We provide two sets of appendices: in Appendix A there is a list of 34 geometries on 4-element set, and a list of 672 geometries on 5-element set, given by their alignments (convex sets), as well as sets of implications representing the associated closure operator. The meet-irreducible elements of alignments are provided, together with their maximal anti-chains, which allows the computation of each geometry's convex dimension.

Appendix B has representations by circles for all 34 discovered geometries on a 4-element set, as well as 623 geometries on a 5-element set. The numbering corresponds to the order of geometries in the lists featured in Appendix A.

The last part of Appendix B lists 49 geometries which are either cannot be represented or not yet represented, giving their implicational bases, also featured in Appendix A. The list of these geometries
is as follows: 
\[
G4, G7, G11, G12, G14, G15, G18, G21, G23, G26, G27, G33, G34, G35, G39, G43,
\]
\[G45, G46, G47, G49, G54, G56, G57, G60, G69, G70, G74, G84, G87, G89, G94,
\]
\[G95, G96, G105, G114, G115, G122, G129, G132, G134, G143, G147, G153,
\]
\[G161, G175, G206, G211, G235, G351. 
\]

Among these 49 geometries, we show in this paper that the following are not representable by circles:
\[
G7, G11, G12, G18, G21, G23, G34.
\]
Excluding these 7 geometries, as well as earlier example G4 which cannot be represented, we ask:
\begin{question}
Are the 41 remaining geometries on the list above representable by circles on the plane?
\end{question}

\section{Preliminaries}\label{preliminaries}
\subsection{Basic Definitions}
A convex geometry is a special case of a closure system. It can be defined through a closure operator, or through an alignment.
\begin{definition}\label{def:closure}
Let $X$ be a set. A mapping $\varphi \colon 2^X \to 2^X$ is called a \emph{closure operator}, if 
for all $Y, Z \in 2^X$:
\begin{enumerate}[noitemsep]
    \item $Y \subseteq \varphi(Y)$
    \item if $Y \subseteq Z$ then $\varphi(Y) \subseteq \varphi(Z)$
    \item $\varphi(\varphi(Y)) = \varphi(Y)$
\end{enumerate}
A subset $Y \subseteq X$ is \emph{closed} if $\varphi(Y) = Y$. The pair $(X,\varphi)$, where $\varphi$ is a closure operator, is called a \emph{closure system}.
\end{definition}

\begin{definition}\label{def:alignment}
Given any (finite) set $X$, an \textit{alignment} on $X$ is a family
$\mathcal{F}$ of subsets of $X$ which satisfies two properties:
\begin{enumerate} [noitemsep]
    \item $X \in \mathcal{F}$
    \item If $Y, Z \in \mathcal{F}$ then $Y \cap Z \in \mathcal{F}$.
\end{enumerate}
\end{definition}
\noindent Closure systems are dual to alignments in the following sense:\newline

\noindent If $(X, \varphi)$ is a closure system, one can define a family of closed sets $\mathcal{F} :=\{Y \subseteq X : \varphi(Y) = Y\}$. Then $\mathcal{F}$ is an alignment.\newline

\noindent If $\mathcal{F}$ is an alignment, then define $\varphi: 2^X \rightarrow 2^X$ in the
following manner:\\
for all $Y \subseteq X$, let 
$\varphi(Y):=\bigcap  \{Z \in \mathcal{F}
: Y \subseteq Z\}$.
Then $(X, \varphi)$ is a closure system. \newline

\begin{definition}\label{def:cg}
A closure system $(X,\varphi)$ is called a \emph{convex geometry} if
\begin{enumerate}[noitemsep]
    \item $\varphi(\emptyset) = \emptyset$
    \item For any closed set $Y\subseteq X$ and any distinct points $x,y \in X\setminus Y,$ if $x \in
    \varphi(Y \cup \{y\})$ then $y \not\in \varphi(Y \cup \{x\})$.
\end{enumerate}
\end{definition}
The second property in this definition is called the \emph{Anti-Exchange Property}.

\noindent We can use duality between closure operators and alignments to provide another definition of a convex
geometry.
\begin{definition}
\label{cg_alignment}
A closure system $(X,\varphi)$ is a convex geometry iff the corresponding alignment
$\mathcal{F}$ satisfies the following two properties:
\begin{enumerate}[noitemsep]
    \item $\emptyset \in \mathcal{F}$
    \item If $Y \in \mathcal{F}$ and $Y \neq X$ then $\exists a\in X\setminus Y$
    s.t. $Y\cup\{a\} \in \mathcal{F}$.
\end{enumerate}
\end{definition}

\subsection{Convex Dimension}\label{cdim}
This section follows the survey on convex geometries by Edelman and Jamison \cite{EdJa85}.
\begin{definition}
An alignment $\mathcal{F}$ is called a \textit{monotone alignment} if the
sets of $\mathcal{F}$ form a chain under the inclusion order.
\end{definition}
\begin{definition}
Given two alignments $\mathcal{F}, \mathcal{K}$, their \textit{join} $\mathcal{F}\vee \mathcal{K}$ is defined to be
the smallest alignment that contains both $\mathcal{F}$ and $\mathcal{K}$. Explicitly, for alignments on finite set $X$, 
this is the collection of all intersections of sets taken from
$\mathcal{F}, \mathcal{K}$: $\mathcal{F}\vee \mathcal{K}=\{ F\cap K : F \in
\mathcal{F} \text{ and }  K \in \mathcal{K} \}$.
\end{definition}
\noindent A known result in \cite{EdJa85} about joins follows.
\begin{theorem}
Every alignment can be expressed as the join of some collection of monotone
alignments on the same base set.
\end{theorem}
This motivates the following definition.
\begin{definition}
The \emph{convex dimension} of a convex geometry $G$ is the least
number $k$ such that the alignment $\mathcal{F}$ corresponding to
$G$ can be expressed as the join of $k$ monotone alignments.
\end{definition}
\noindent To compute the convex dimension of a convex geometry, we can examine
antichains of meet-irreducibles.
\begin{definition}
A \textit{meet-irreducible} in an alignment
$\mathcal{F}$ is a closed set $Y \in \mathcal{F}$ for which there exists $Z \in \mathcal{F}$ such that
\begin{enumerate}[noitemsep]
    \item $Y \subset Z$
    \item If $Y\subset W$, for some closed set $W$, then $Z\subseteq W$.
\end{enumerate}
\end{definition}

\noindent This is equivalent to saying that if $Y=\bigcap_{i\in I} Y_i$ where $Y_i\in \mathcal{F}$ for all $i$, then $Y=Y_i$ for some $i\in I.$ \newline

\noindent Note that by definition \ref{cg_alignment} of a convex geometry, we can additionally claim that $|Z\setminus Y|=1$. 

\begin{definition}
In a partially ordered set, an \emph{antichain} is a collection of pairwise
non-comparable elements.
\end{definition}
\noindent Set inclusion forms a partial order on the elements of an alignment, and this
partial order allows us to talk about antichains of elements of an alignment.
The following result is due to Edelman and Saks \cite{EdSa88}.
\begin{theorem}\label{thm:cdim}
Let $G$ be a convex geometry and $\mathcal{F}$ be the corresponding alignment.
The convex dimension of $G$ is equal to the largest size antichain of
meet-irreducibles in $\mathcal{F}$.
\end{theorem}

\subsection{Convex hull operator}

A particular example of a closure operator on a set is the convex hull operator, where the base set $X$ is a set of points in Euclidean space $\mathbb{R}^n$. For the goals of this paper we use only the plane, i.e. the space $\mathbb{R}^2$.

\begin{definition}
A set $S$ in $\mathbb{R}^2$ is called  \emph{convex} if for any two points $p,q \in S$, the line segment connecting $p$ and $q$ is also contained in $S$. 
\end{definition}

See figure below for an illustration.
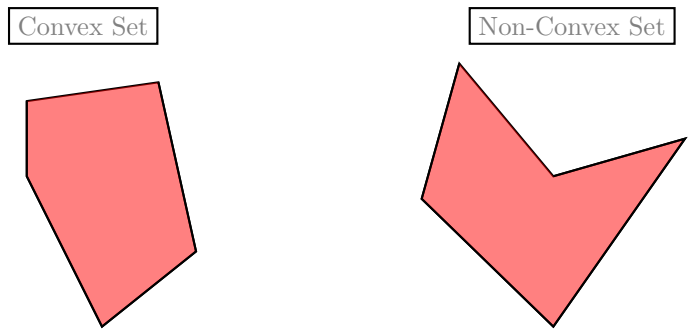
\begin{figure}[H]
    \centering
  \begin{tikzpicture}[thick,fill opacity=0.5]
      \draw (-4,3) - - (-4,2) - - (-3,0) - - (-1.75,1) - - (-2.25,3.25) - - cycle ;
      \filldraw[fill=red] (-4,3) - - (-4,2) - - (-3,0) - - (-1.75,1) - - (-2.25,3.25) ;
  \node[draw] at (-3.25,4) {Convex Set};
      
     \draw (3,2) - - (4.75,2.5) - - (3,0) - - (1.25,1.7) - - (1.75,3.5) - - cycle ;
      \filldraw[fill=red] (3,2) - - (4.75,2.5) - - (3,0) - - (1.25,1.7) - - (1.75,3.5) ;
      
        \node[draw] at (3.25,4) {Non-Convex Set};
    
   \end{tikzpicture}  
   \caption{Example of a set which is convex and a set which is not convex.}
\end{figure} 

\begin{definition}
Given set $S$ of points in $\mathbb R ^ 2$, the convex hull of $S$, in notation $\CH(S)$, is the intersection of all convex sets in $\mathbb R ^2$ which contain $S$.
\end{definition}

From the illustration, if $S=\{A,B,C,D,E,F,G\}$, a set of of 7 points on the plane, $\CH(S)$ is a convex shape with vertices $A,C,D,F,G$ in Figure \ref{fig:CH}.

\begin{figure}[H]
    \centering
  \begin{tikzpicture}[thick,fill opacity=0.5]
      \draw  (-2,1) - - (-.5,2.8) - -(.5,3) - - (3,2.5) - -(2,0) - - cycle;
    \filldraw[fill=red]  (-2,1) - - (-.5,2.8) - -(.5,3) - - (3,2.5) - -(2,0);

   \filldraw [black, thick] (-2,1) circle[radius=0.075] node[anchor=south] {$A$};
    \filldraw [black, thick] (1,1.75) circle[radius=0.075] node[anchor=south] {$B$};
    \filldraw [black, thick] (2,0) circle[radius=0.075] node[anchor=south] {$C$};
     \filldraw [black, thick] (-.5,2.8) circle[radius=0.075] node[anchor=south] {$D$};
    \filldraw [black, thick] (1,1) circle[radius=0.075] node[anchor=south] {$E$};
    \filldraw [black, thick] (.5,3) circle[radius=0.075] node[anchor=south] {$F$};
    \filldraw [black, thick] (3,2.5) circle[radius=0.075] node[anchor=south] {$G$};

   \end{tikzpicture}  
  \caption{The convex hull $\CH(ABCDEFG)$}
  \label{fig:CH}
  \end{figure}
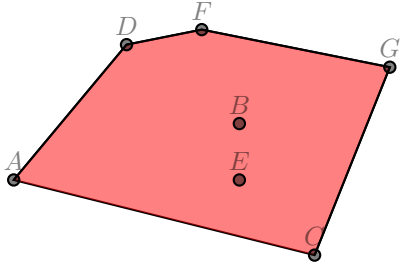  

The following definition allows us to introduce a closure operator induced by the convex hull, on any set of points $X$ on the plane. 

\begin{definition}
Let $X$ be a finite set of points in $\mathbb{R}^2$. Define an operator $\ch : 2^X \rightarrow 2^X$ as follows:
\[
 \ch(Y) = \CH(Y)\cap X,
\]
for any $Y\in 2^X$.
\end{definition}

It is straightforward to check that $\ch$ is a closure operator. For example, $\ch(ACDFG)=S$ in Figure \ref{fig:CH}. Moreover, $\ch$ satisfies the Anti-Exchange Property. Therefore, $(X,\ch)$ is a convex geometry.\\

Finally, for the purposes of the current paper we want to recall the definition of the convex hull operator for circles introduced in Cz\'edli \cite{C14}.

If $x$ is a circle on the plane, then by $\tilde{x}$ we denote a set of points belonging to $x$. We allow a circle to have a radius $0$, in which case it is a point.

\begin{definition}
Let $X$ be a finite set of circles in $\mathbb{R}^2$. Define the operator $\ch_c : 2^X \rightarrow 2^X$, a convex hull operator for circles, as follows:
\[
\ch_c(Y) =\{x \in X: \tilde{x}\subseteq \CH(\bigcup_{y \in Y}\tilde{y})\},
\]
for any $Y\in 2^X$.
\end{definition}
For example, $a, e \in \ch_c(b,c,d)$ on Figure \ref{fig:ch1}, while $a \not \in \ch_c(b,c,d)$, $e \in \ch_c(b,c,d)$ in Figure \ref{fig:ch2}.

%

\begin{figure}
\begin{center}
    \begin{subfigure}[a]{0.4\textwidth}
        \centering
       \includegraphics[width=\textwidth]{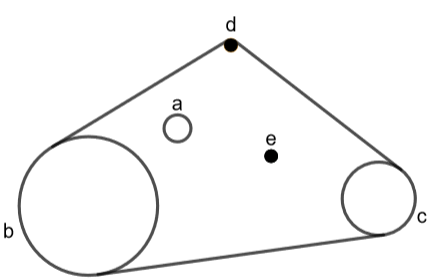}
      \caption{Circle $a$ is in the convex hull of $b,c,d$}
           \label{fig:ch1}
    \end{subfigure}
\hspace{1cm}
    \begin{subfigure}[a]{0.4\textwidth}
        \centering
        \includegraphics[width=\textwidth]{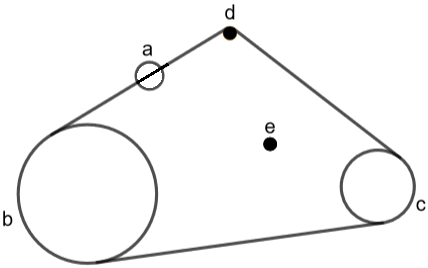}
        \caption{Circle $a$ is not in the convex hull of $b,c,d$}
            \label{fig:ch2}
    \end{subfigure}
\end{center}
\caption{Examples of convex hulls}
\label{fig:convex hull example 2}
\end{figure}

\section{Creating the list of geometries on 4- and 5-element sets} \label{code} 

Our description of convex geometries represented by 4 or 5 circles on the plane relies on the OEIS \cite{oeis} (On-line Encyclopedia of Integer Sequences) submission on the number of non-isomorphic antimatroids by Przemysław Uznański \cite{Uz13}.

From a combinatorial point of view, a family $\mathcal{F}$ of subsets of (finite) $X$ is called an \textit{antimatroid} if and only if:

\begin{enumerate}
    \item $\mathcal{F}$ is closed under unions, i.e. $\forall  Z,Y \in \mathcal{F}, Z\cup Y \in \mathcal{F}$
    \item $\mathcal{F}$ is non-empty accessible, i.e $\forall Y \in \mathcal{F}$ either $Y = \emptyset$ or $\exists\; x \in Y$ such that $Y\setminus x \in \mathcal{F}$.
\end{enumerate}

\noindent A close look at the properties reveals that there exists a bijective map $f$, between antimatroids and convex geometries given by:

\begin{equation}
    f(\mathcal{F}) = \{ Y^c \colon Y \in \mathcal{F} \}
\end{equation}
Here $Y^c = X\setminus Y$, the set-theoretic complement of the subset $Y\subseteq X$.

\noindent Uzna\'nski in \cite{Uz13} generated a list of all non-isomorphic antimatroids on base sets of order 3, 4, 5 and 6 using an extension property on antimatroids combined with a reverse tree traversal algorithm. The theoretical background of the algorithm was developed in Kempner and Levit \cite{KeLe03}. 

Adapting Uzna\'nski's original codebase, all convex geometries can be easily generated using the bijection noted above. The antimatroids and their corresponding geometries are represented as elements of $2^{2^X}$ using their associated alignments as described by Definition \ref{def:alignment}. For efficiency and portability however, the alignments are physically stored, and listed in the Appendices A1 and A2, as 32-bit integers through a programming technique known as \textit{bitmasking} \cite{MaskG}. 

Given a 32-bit representation $b_{31}b_{30} \dots b_1b_0$ and $X = \{ x_1 \dots x_n \}$ where $n = 3, 4, 5$, one may obtain the convex sets, $F$, by writing $S = \{ i \ \colon \ b_i = 1 \}$. Each convex set $F_k$ is then given by

\begin{equation}
    F_k = \{ x_i \ \colon \ c_i = 1, \  c_5 \dots c_1 = S_k \in S\}
\end{equation}
where $c_i$ are the bits of $S_k$.

\begin{example}
\end{example}
 For a base set of order 3, represented as $\{a,b,c\},$ assign the zeroth, first and second bits to $a,b$ and $c$ respectively. Then $\{a\}$ is represented as 001(1), $\{b\}$ is represented as 010(2), $\{a,b,c\}$ is represented as 111(7), $\emptyset$ as 000(0) and so on.
\\
\noindent It is easy to check that the decimal equivalents of the binary representation of subsets form the set $\{0,1,2\dots2^n-1\}$ where $n$ is the order of the base set. A family $F$ has its $n^{th}$ bit set (value 1), if and only if the subset is represented by $n$ $\in F.$ 

\begin{example}
\end{example}
For example, for a base set of order 3, the family $ F=\{\emptyset,\{a\},\{a,b\},\{a,b,c\}\}$ is represented as 10001011(139) because $\emptyset \to 0,\{a\} \to 1,\{a,b\}\to 3,\{a,b,c\} \to 7$, and hence the zeroth, first, third and seventh bits are set. 
\\

\noindent A useful perspective on a convex geometry is provided by its \textit{implications.} Given a convex geometry $(X, \varphi)$ with alignment $\mathcal{F}$, and ordered pair $(A,B)$, where $A, B \in 2^X$, also denoted 
$A \to B$, is an implication if:
\begin{equation}
    B \subseteq \varphi(A)=\bigcap \{ Y \colon A \subseteq Y, \  Y \in F\} 
\end{equation}

\noindent 

A set of implications fully defining closure operator $\varphi$ is often referred to as \emph{implicational basis} of a closure system. The recent exposition on implicational bases is given in Adaricheva and Nation \cite{AN16I}.

\noindent  Because implications are very useful for representing geometries, they were generated for each geometry in addition to its convex sets. However, some implications are redundant with respect to others. For example, $ab \to c$ is redundant if $a \to c$. Hence work was also done in attempt to remove such redundancies. 

In practice we only removed pairwise redundancies, but an \emph{optimum} implication list could be found using a brute force approach (however, the problem of finding such an optimum implication list for general closure system  is 
NP-hard, see \cite{ADS,MW}). The pairwise reduction was achieved by declaring implication $A \to B$ redundant if there exists an implication $C \to D$ such that $C \subseteq A$ and $B \subseteq D$.

To obtain a list of only convex geometries with convex dimension less than or equal to 5, as per Problem 1 in \cite{AdBo19}, calculation of each geometry's dimension was also necessary. As described in Section \ref{cdim}, one may compute convex dimension by starting with the meet-irreducible convex sets, which can be identified by the convex sets $Y$ that have only one  convex superset $Z$ such that $|Z\setminus Y|=1$. From there, one can build a maximal antichain of said irreducibles using recursive techniques. And by Theorem \ref{thm:cdim}, the convex dimension is the cardinality of this maximal antichain.

\section{Representing the list}\label{conquer}

The methods described in section \ref{code} were used to generate 34 distinct geometries on a set of 4 elements and 672 distinct geometries on a set of 5 elements. As described in that section, a convex geometry on $n$ elements can be represented by $n$ circles on the plane, with implication corresponding to containment within convex hulls, see more details in subsection \ref{check}. 
We are interested in the representability of each of the geometries one generates in this manner. To wit, we want to establish for each of the 34 4-element  geometries a representation using 4 circles on the plane; similarly, we want to establish for each of 672 5-element geometries a representation using 5 circles on the plane.

Representations of all 34 4-element geometries are presented in Appendix B1. However, we found that some of the 5-element geometries resisted representation. In particular, of 672 geometries generated, we constructed 623 accurate representations, which are given in Appendix B2. Additionally, 7 of the remaining geometries we found to be provably non-representable. 

In this section we address the techniques used to produce representations in Appendices B1-2.

\subsection{Representation of geometries on 5-element set based on representations on 4-element or 3-element set}\label{chunks}

In the process of representing the 672 geometries on a 5-element base set, we found that the representations of those with a unique atom (${a},{b},{c},{d}$, or ${e}$) and those with a unique coatom (${abcd},{abce},{abde}$, or ${bcde}$) could be generated using our existing representations of geometries on a 4-element base set. The former were generated by creating a non-empty intersection between the four elements and representing the fifth element with a point contained in this intersection, and the latter by representing the fifth element with a circle that completely contained the other four elements. We used a coded search function to generate lists of geometries with a unique atom and a unique coatom from our master list, to facilitate the identification of the representations we created.

We generated other groups of 5-element representations by doubling circles in the 4-element representations and slightly off-setting them, or by placing a fifth element in just one of the existing four. To identify these representations in the master list, a program was created to search for geometries isomorphic to the input.

\subsection{Checking correctness of representations}\label{check}
In this section we describe the method of checking that a convex geometry given by its implications, as in Appendices A1-2, is isomorphic to a geometry represented by circles on the plane. 

We recall from section \ref{code} that a closure operator $\varphi$ can be recorded in the set $\Sigma$ of implications, where $Y\to Z$ means $Z\subseteq \varphi(Y)$. Then the corresponding alignment of $\varphi$-closed sets comprises all subsets of $X$ that \emph{respect} all implications in $\Sigma$: if $(Y\to Z)\in \Sigma$ and $Y\subseteq W$ for set $W$ in alignment, this would imply $Z\subseteq W$.

Convex geometries on our list were given by both an alignment $\mathcal{F}$ on set $X$ and a set of implications $\Sigma$. In order to produce a representation by circles, each element of $X$ was represented by a circle, and the positioning of circles in the plane would define a closure operator $ch_c$, the convex hull operator for circles. A geometry generated by these circles would be isomorphic to $(X,\mathcal{F})$ if the alignments match, i.e., if the alignment $\mathcal{G}$ generated by $ch_c$ consists of the same subsets as $\mathcal{F}$. The following two statements formed the basis for such matching.
        
\begin{proposition}
Let $(X,\mathcal{F})$ be  an alignment defined by a set of implications $\Sigma$. Let $\psi$ be any closure operator defined on $X$ and the corresponding alignment of $\psi$-closed sets be $\mathcal{G}$. If for every $(Y\to Z)\in \Sigma$ one has $Z\subseteq \psi(Y)$, then $\mathcal{G}\subseteq \mathcal{F}$. 
\end{proposition}
\begin{proof}
Indeed, take any set $W \in \mathcal{G}$. We need to show that $W \in \mathcal{F}$, i.e., it respects all implications from $\Sigma$. So take $Y\to Z$ in $\Sigma$ and let $Y\subseteq W$. Since by assumption we have $Z\subseteq \psi(Y)$, and $W$ is $\psi$-closed, we have $Z\subseteq W$ as desired. 
\end{proof}
\noindent Concretely, we would check that for each implication $Y\to Z$ from the set of implications defining the given convex geometry, all circles in $Z$ are in the convex hull of circles from $Y$.
        
\begin{proposition}
If all meet-irreducible sets from alignment $\mathcal{F}$ belong to $\mathcal{G}$, then $\mathcal{F}\subseteq \mathcal{G}$.
\end{proposition}
\begin{proof}
Indeed, every set in $\mathcal{F}$ is the intersection of meet-irreducible sets. Therefore, they all belong to $\mathcal{G}$.
\end{proof}
\noindent In practice, we would check that all circles belonging to the meet-irreducible set of a given convex geometry form a convex set of circles.
\noindent Combining the two practical procedures following from the propositions above, we could conclude that the given convex geometry is isomorphic to our convex geometry of circles.
\\

\subsection{GeoGebra as a toolkit for representations}\label{Geo}
At the beginning of our work, creating a representation for a geometry entailed drawing out the circles by hand, and then manually representing the drawing in LaTeX. One had to take care to preserve the convex hulls as precisely as possible and not generate unwanted implications. 

This became especially cumbersome on certain geometries that were possible but required a very precarious setup, where even the slightest shift in the position of a circle would make the representation incorrect. We started using the software
GeoGebra as a tool for representing Geometries, which allowed for the easy rearrangement of circles and points to one's liking. Eventually we 
came up with a template GeoGebra file \cite{ggb} that allowed users to adjust not only the position and radii of the circles, but also to draw the convex hulls of all possible combinations of circles. This made it easier to check if a representation was actually correct or not. 

Using GeoGebra helped in checking if a representation of a geometry was correct, but once represented, they still had to be remade in \LaTeX using the TikZ package. It was soon discovered that GeoGebra had a built in way to export images to TikZ code, and we utilized and enhanced this code
to generate the desired representation.

\subsection{Google Colab toolkit}\label{Google}
In order to streamline the process of exporting representations of geometries to \LaTeX, we generated
a Python script to automatically generate TikZ code to transfer to \LaTeX.
The centers and radii of the circles were shown in GeoGebra, so the user only had to copy those numbers into the script. 

Simultaneously, we created another 
script in a Google Colab notebook that would parse through the TikZ code generated by GeoGebra, extract all of the necessary information (such as circle location and radii) and then output useable TikZ code. This new code, however, had issues with normalizing the size of the circles to fit nicely on the page, and so features of both scripts were combined in the Google Colab notebook. The end result was code that would accept a text file from GeoGebra and reliably output useable TikZ code for \LaTeX \cite{colab}.

\section{Triangle Property and Impossible geometries}\label{math}

In this section we will be proving that several geometries on a 5 element set are not representable by circles on the plane.
\subsection{Tight implication on circles}
The following statement is a modification of Lemma 5.1 in \cite{AdBo19}.

\begin{lemma}\label{3 circles}
Let $a,b,c$ be three circles on the plane, where none is inside another. 
Then there are three different types of configurations of three circles $a,b,c$, up to re-labeling:

\begin{itemize}
    \item[(1)] $\tilde{b}\subseteq \CH(\tilde{a}\cup\tilde{c})$;
    \item[(2)] $\CH(\tilde{a}\cup\tilde{b}\cup\tilde{c}) = \CH(\tilde{a}\cup\tilde{b})\cup\CH(\tilde{b}\cup\tilde{c})$
    as in Figures \ref{fig:WeakCarousel2} or \ref{fig:WeakCarousel2Case};
    \item[(3)] $\CH(\tilde{a}\cup\tilde{b}\cup\tilde{c})$
    is inscribed into a triangle as in Figure \ref{fig:WeakCarousel3}, or forms a shape as in Figure \ref{fig:WeakCarousel4}.
\end{itemize}
\end{lemma}
    
   \begin{multicols}{2}
   \begin{figure}[H]\centering
\includegraphics[width=0.33\textwidth]{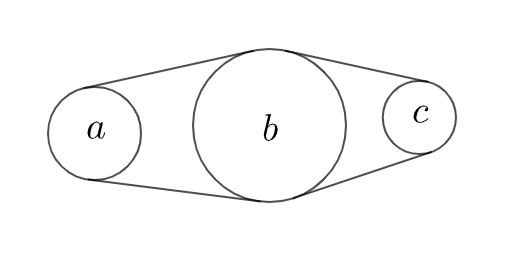}  
\caption{ }
  \label{fig:WeakCarousel2}
\end{figure}

\columnbreak

   \begin{figure}[H]\centering
\includegraphics[width=0.3\textwidth]{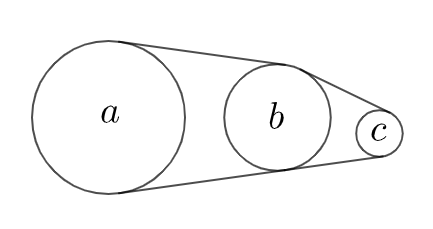} \caption{ }
  \label{fig:WeakCarousel2Case}
\end{figure}

\end{multicols}

\begin{multicols}{2}
   \begin{figure}[H]\centering
\includegraphics[width=0.175\textwidth]{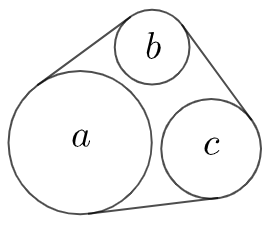}  
\caption{ }
  \label{fig:WeakCarousel3}
\end{figure}

 \columnbreak
 
   \begin{figure}[H]\centering
\includegraphics[width=0.145\textwidth]{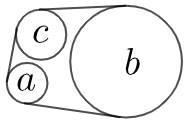}  
\caption{ }
  \label{fig:WeakCarousel4}
\end{figure}
\end{multicols}
We want to reformulate this classification of three circle configurations in the terms of the boundaries of convex hulls. This allows to avoid an overlap between cases (2) and (3) of Lemma \ref{3 circles}.
\begin{lemma}\label{borders}
The boundary of 
$\CH(\tilde{a}\cup\tilde{b}\cup\tilde{c})$
for any three circles $a,b,c$, none of which includes another, comprises
\end{lemma}
\begin{itemize}
    \item [(I)] two tangent segments and two arcs of circles, when configuration (1) from Lemma \ref{3 circles} occurs;
    \item[(II)] four tangent segments and four arcs, as in Figure \ref{fig:WeakCarousel2} or \ref{fig:WeakCarousel2Case}, for most configurations in (2) from Lemma \ref{3 circles}; 
    \item[(III)] three tangent segments and three arcs, as in Figure \ref{fig:WeakCarousel3} and Figure \ref{fig:WeakCarousel4}, for configurations in (3) from Lemma \ref{3 circles}, or, in an extreme case of configuration (2).
\end{itemize}
\begin{proof}

Case (1) is distinguished from others by the property that  
one of circles does not have an arc on the boundary of $\CH(\tilde{a}\cup\tilde{b}\cup\tilde{c})$.

In case (2) one of the circles crosses both outer tangent lines of two other circles and 
has two points of intersection with each of those tangent lines.

In case (3), each circle is contained in one closed semiplane formed by one of outer tangent lines of two others. Therefore, a segment of that tangent line will be a part of the boundary of 
$\CH(\tilde{a}\cup\tilde{b}\cup\tilde{c})$. The configuration in Figure \ref{fig:WeakCarousel4} is distinguished by having an arc corresponding to central angle larger than $\pi$ from one of the circles as a part of the border of 
$\CH(\tilde{a}\cup\tilde{b}\cup\tilde{c})$. 
\end{proof}

\begin{definition}
An implication $Y\to u$, $Y\subseteq X, u\in X$, is called \emph{tight}, if implications $(Y\setminus z) \to u$ do not hold, for all $z \in Y$.
\end{definition}

\begin{lemma}\label{non-tight}
For an implication $Y\to u$, $Y\subseteq X, u\in X$, if there exists $\ y\in Y$ such that $(Y\setminus y) \to y$, then implication $Y\to u$ cannot be tight.
\begin{proof}
Let $\varphi$ be the closure operator. If there exists $ \ y\in Y$ such that $(Y\setminus y) \to y$, then $y \in \varphi((Y\setminus y))$. Also since for any $A, \ A \subseteq \varphi(A)$, we have $Y \subseteq \varphi((Y\setminus y))$. Now applying the closure operator to both sides, $\varphi(Y) \subseteq \varphi((Y\setminus y))$, as $\varphi(\varphi(A)) = \varphi(A)$. Since $Y\to u$, $u \in \varphi(Y)$ and $u \in \varphi((Y\setminus y))$. Thus $(Y\setminus y) \to u$, and hence $Y\to u$ cannot be tight. 
\end{proof}
\end{lemma}

\begin{lemma}\label{no12}
If $abc \to e$ is any  tight implication  that holds for circles $a,b,c,e$, then $a,b,c$ cannot form a configuration in (1) or (2) of Lemma \ref{3 circles}, for any labeling of circles by 
$a$, $b$, $c$.
\end{lemma}
\begin{proof}
Indeed, if $a,b,c$ form a configuration in (1), then the implication $abc \to e$ cannot be tight due to Lemma \ref{non-tight}.
For configuration (2) 
we note that for any circle $e$ in $\ch_c(a, b, c)$ we have either $e \in \ch_c(a, b)$ or $e \in \ch_c(b, c)$. Note that the same may not be true for a shape $e$ different from circle.
\end{proof}

\begin{corollary}\label{tight}
If $abc \to e$ is any  tight implication  that holds for circles $a,b,c,e$, then the boundary of 
$\CH(\tilde{a}\cup\tilde{b}\cup\tilde{c})$ 
comprises three tangent segments and three arcs. 
\end{corollary}
\begin{proof}
By Lemma \ref{no12}, only configuration (3) from Lemma \ref{3 circles} may occur, which implies only case (III) of Lemma \ref{borders} occurs.
\end{proof}

\subsection{Triangle Property}

In this section we establish a simple property that holds for circles satisfying a tight implication. This will result in establishing that several convex geometries on a 5-element set are not representable by circles.

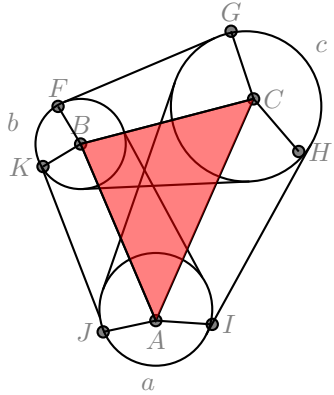
\begin{figure}[H]
    \centering
  \begin{tikzpicture}[thick,fill opacity=0.5]
    \node[draw,circle,xshift=2.2cm,yshift=.5cm,minimum size=20mm,outer sep=0] (a) {};
    \node[draw,circle,minimum size=12mm,outer sep=0] (b) {};
     \node[draw,circle,xshift=1cm,yshift=-2.2cm,minimum size=15mm,outer sep=0] (c) {};
     
     \filldraw [black, thick] (2.3,.6) circle[radius=0.075] node[anchor=west] {$C$};
     \filldraw [black, thick] (0,0) circle[radius=0.075] node[anchor=south] {$B$};
      \filldraw [black, thick] (1,-2.35) circle[radius=0.075] node[anchor=north] {$A$};
      
      \filldraw [black, thick] (2,1.5) circle[radius=0.075] node[anchor=south] {$G$};
      \filldraw [black, thick] (2.9,-0.1) circle[radius=0.075] node[anchor=west] {$H$};
      \filldraw [black, thick] (1.75,-2.4) circle[radius=0.075] node[anchor=west] {$I$};
      \filldraw [black, thick] (.3,-2.5) circle[radius=0.075] node[anchor=east] {$J$};
      \filldraw [black, thick] (-0.3,.5) circle[radius=0.075] node[anchor=south] {$F$};
      \filldraw [black, thick] (-.5,-0.3) circle[radius=0.075] node[anchor=east] {$K$};
      
      \draw (2,1.5) -- (2.3,.6cm);
      \draw (2.9,-0.1) -- (2.3,.6cm);
      \draw (1.75,-2.4) -- (1,-2.35cm);
      \draw (.3,-2.5) -- (1,-2.35cm);
      \draw (-.3,.5) -- (0,0cm);
      \draw (-.5,-.3) -- (0,0cm);
     
    \draw (tangent cs:node=b,point={(a.south)},solution=2) -- (tangent cs:node=a,point={(b.south)});
    \draw (tangent cs:node=b,point={(a.north)},solution=1) -- (tangent cs:node=a,point={(b.north)},solution=2);
    
    \draw (tangent cs:node=b,point={(c.south)},solution=2) -- (tangent cs:node=c,point={(b.south)});
    \draw (tangent cs:node=b,point={(c.north)},solution=1) -- (tangent cs:node=c,point={(b.north)},solution=2);
    
   \draw (tangent cs:node=a,point={(c.west)},solution=2) -- (tangent cs:node=c,point={(a.west)});
    \draw (tangent cs:node=a,point={(c.east)},solution=1) -- (tangent cs:node=c,point={(a.east)},solution=2);
    
      \draw[black] (3.2, 1.3)node {$c$};
      \draw[black] (-0.9, 0.3)node {$b$};
      \draw[black] (0.9, -3.2)node {$a$};
      
      \draw (2.3,.6) - - (0,0) - - (1,-2.35) - - cycle ;
      \filldraw[fill=red] (2.3,.6) - - (0,0) - - (1,-2.35) ;
      \end{tikzpicture}
      
\caption{Triangle Property}
\label{fig:CenterTriangle}
\end{figure}  

\begin{lemma}[The Triangle Property]\label{triangle}
If $a,b,c,$ and $e$ are circles in the plane with centers $A,B,C,$ and $E$ respectively, and $abc\to e$ is a tight implication, then $E$ must lie in the interior of triangle formed by $A$,$B$ and $C$.
\end{lemma}

\begin{proof}
If implication $abc\to e$ is tight, then by Lemma \ref{no12} circles $a,b,c$ cannot form any configuration from (1) and (2) of Lemma \ref{3 circles}.

Moreover, since $abc\to e$ is tight, the centers of $a,b,c$ cannot be on a line, because in that case one obtains one of configurations (1) or (2) in Lemma \ref{3 circles}. Therefore, the centers form a triangle.

By Corollary \ref{tight}, the boundary of $\CH(\tilde{a}\cup\tilde{b}\cup\tilde{c})$ comprises three tangent segments and three arcs, as in Figure \ref{fig:CenterTriangle}.
The convex hull of three circles $a,b,c$ splits into two parts:
\begin{enumerate}[(I)]

\item the interior $int(\triangle ABC)$ of triangle  formed by the centers of the circles; 

\item the shape 
$S=
\CH(\tilde{a}\cup\tilde{b}\cup\tilde{c})\setminus int(\triangle ABC)$
.
\end{enumerate}
Moreover, $S$ splits into the union of the following shapes: 
\begin{enumerate}[(I)]

\item trapezoids formed between pairs of circles, each of which includes a segment connecting the centers of the two circles, two radii, and the circles' common tangent segment; 
\item sectors of circles bordered by the radii which are the borders of trapezoids from (I).
\end{enumerate}
If circle $e$ is inside of the convex hull of $a,b,c$, and its center $E$ is not in the interior of $\triangle ABC$, then it is in $S$. Thus it is either in one of the trapezoids or in one of the sectors. 

Without loss of generality, suppose that $E$ is inside of one of the sectors of circle $a$. Then the distance $\sigma$ from $E$ to the border of circle is $\leq r_a$. Therefore, $r_e\leq r_a$ and $e\in \ch_c(a)$. Here we use $r_x$ to denote the radius of circle $x$.

Suppose now that $E$ is in one of the trapezoids. Without loss of generality, suppose that $E$ is in $BCGF$ in Figure
\ref{fig:CenterTriangle}, where $FG$ is the part of the border of $\CH(\tilde{a}\cup\tilde{b}\cup\tilde{c})$, and $\sigma$ is the distance from $E$ to segment $FG$. Then $r_e\leq \sigma$ and $e \in \ch_c(b,c)$. 

Therefore, the center of $e$ for which $abc\to e$ is tight cannot be in either the trapezoids or sectors, and thus it must be in interior of $\triangle ABC$.
\end{proof}

\noindent The following statement is a generalization of the Triangle Property.

\begin{corollary}\label{ncircles}
If $n>2$, $a_1,a_2,a_3,\cdots,a_n,e$ are circles in the plane with centers $A_1,A_2,\cdots,A_n,E$ respectively, and $a_1a_2a_3\cdots a_n\to e$ is tight then $E\in\CH(A_1,A_2,\cdots,A_n)$.
\end{corollary}
\begin{proof}
The proof proceeds similar to the case of three circles and can be omitted.

\end{proof}

\begin{corollary}\label{wedge}
Convex geometries representable by circles cannot have three tight implications $abc\to e$, $abd\to e$, $acd\to e$.
\end{corollary}
\begin{proof}
Indeed, consider $A,B,C,D,E$, the centers of circles that might represent this geometry. Connect $A$ with $B,C$ and $D$. Then the plane splits into three angles at vertex $A$. Choose two of smallest angles among these three. Without loss of generality, assume that the smallest are $\angle CAB$ and $\angle BAD$. Then each of the two smallest angles are less than $\pi$, thus, line $(AB)$ splits the plane into two semi-planes, so that $C$ and $D$ are in distinct
semi-planes (see illustration in Figure \ref{fig:Triangle3}).

\begin{figure}[H]
    \centering
  \begin{tikzpicture}[thick,fill opacity=0.5]
    \node[draw,circle,xshift=2.2cm,minimum size=15mm,outer sep=0] (b) {};
       \node[draw,circle,minimum size=15mm,outer sep=0] (c) {};
       \node[draw,circle,xshift=2.25cm,yshift=-2.2cm,minimum size=20mm,outer sep=0] (a) {};
       \node[draw,circle,xshift=4.5cm,yshift=0.8cm,minimum size=15mm,outer sep=0] (d) {};
       
     \filldraw [black, thick] (0,0) circle[radius=0.075] node[anchor=south] {$C$};
     \filldraw [black, thick] (2.2,0) circle[radius=0.075] node[anchor=south] {$B$};
     \filldraw [black, thick] (4.5,0.8) circle[radius=0.075] node[anchor=south] {$D$};
     \filldraw [black, thick] (2.3,-2.2) circle[radius=0.075] node[anchor=north] {$A$};
     
     \draw (0,0) -- (2.3,-2.2cm);
     \draw (4.5,0.8) -- (2.3,-2.2cm);
     \draw (2.2,0) -- (2.3,-2.2cm);
     
    \draw (tangent cs:node=b,point={(a.west)},solution=4) -- (tangent cs:node=a,point={(b.west)});
    \draw (tangent cs:node=b,point={(a.east)},solution=0) -- (tangent cs:node=a,point={(b.east)},solution=2);
    
    \draw (tangent cs:node=b,point={(c.north)},solution=2) -- (tangent cs:node=c,point={(b.north)});
    \draw (tangent cs:node=b,point={(c.south)},solution=1) -- (tangent cs:node=c,point={(b.south)},solution=2);
    
    \draw (tangent cs:node=a,point={(c.east)},solution=2) -- (tangent cs:node=c,point={(a.east)});
    \draw (tangent cs:node=a,point={(c.west)},solution=1) -- (tangent cs:node=c,point={(a.west)},solution=2);
    
    \draw (tangent cs:node=a,point={(d.south)},solution=2) -- (tangent cs:node=d,point={(a.south)});
    \draw (tangent cs:node=a,point={(d.north)},solution=1) -- (tangent cs:node=d,point={(a.north)},solution=2);
    
    \draw (tangent cs:node=b,point={(d.south)},solution=2) -- (tangent cs:node=d,point={(b.south)});
    \draw (tangent cs:node=b,point={(d.north)},solution=1) -- (tangent cs:node=d,point={(b.north)},solution=2);
    
    \draw[black] (1.1, -2.8)node {$a$};
    \draw[black] (-0.5, .8)node {$c$};
    \draw[black] (3.7, 1.6)node {$d$};
    \draw[black] (1.7, 1.)node {$b$};
        
    \end{tikzpicture}
\caption{Corollary \ref{wedge}}
\label{fig:Triangle3}
\end{figure}

By Lemma \ref{triangle}, $E$ must be in the interior of $\triangle ABC$ as well as in the interior of $\triangle ABD$, and since those interiors do not intersect, there is no possible location for $E$.
\end{proof}

{\bf This gives us six geometries on the list that are not representable: G7, G11, G12, G18, G21 and G34. Note that G7, G21 and G34 have cdim=4.}

\begin{corollary}
Convex geometries representable by circles cannot have three tight implications $abc\to d$, $acd\to e$, $bcd\to e$.
\end{corollary}
\begin{proof}
Assume that the geometry does have a representation as a geometry of circles.  The proof then follows from the use of Lemma \ref{triangle} two times. First, since $abc \rightarrow d$ is a tight implication, $D$ must lie in the interior of triangle formed by vertices $A$, $B$ and $C$ (which are the centers of $a,b,c$). Moreover, we have two tight implications $acd \rightarrow e$, $bcd \rightarrow e$. However, by Lemma \ref{triangle} again, we would need $E$ to be in the interior of both $\triangle ACD$ and $\triangle BCD$, which is not possible. 

\end{proof}
{\bf This gives us that the additional geometry G23 cannot be represented by circles on the plane.}

\section{Acknowledgments.}
The project was initiated by several faculties from across the US led by Adam Sheffer (Baruch College, NY). These include Benjamin Brubaker (University of Minnesota), Victor Reiner (University of Minnesota), Patrick Devlin (Yale University), Yunus Zeytuncu (University of Michigan, Dearborn), Steven J. Miller (Williams College, MA) and Alexandra Seceleanu (University of Nebraska, Lincoln). Without the support and initiative of these colleagues and the unique circumstances of 2020, this project would not have started. More details about PolyMath-2020 can be in found in \cite{Lemons}.\\ \\
\textit{Address for correspondence:}\\
Kira Adaricheva, Department of Mathematics, Hofstra University, Hempstead NY, 11549\\
\textit{Email address:} \href{mailto:kira.adaricheva@hofstra.edu}{\texttt{Kira.Adaricheva@Hofstra.edu}}



\newpage

\appendix
\appendixpage
\addappheadtotoc

\section{}


\subsection{Description of size 4 geometries} \chead{\sc Appendix A.1: Description of size 4 geometries} \begin{verbatim}
1.
id: 65535
meet-irreducibles: abc abd acd bcd 
max antichain: abc abd acd bcd 
dimension: 4

2.
id: 65407
abc -> d
meet-irreducibles: ab ac bc abd acd bcd 
max antichain: ab ac bc 
dimension: 3

3.
id: 65399
ab -> d
meet-irreducibles: ac bc abd acd bcd 
max antichain: ac bc abd 
dimension: 3

4.
id: 65367
ac -> d
ab -> d
meet-irreducibles: a bc abd acd bcd 
max antichain: bc abd acd 
dimension: 3

5.
id: 63351
ab -> cd
meet-irreducibles: ac bc ad bd acd bcd 
max antichain: ac bc ad bd 
dimension: 4

6.
id: 65365
a -> d
meet-irreducibles: bc abd acd bcd 
max antichain: bc abd acd 
dimension: 3

7.
id: 65303
bc -> d
ac -> d
ab -> d
meet-irreducibles: a b c abd acd bcd 
max antichain: a b c 
dimension: 3

8.
id: 63319
ac -> d
ab -> cd
meet-irreducibles: a bc ad bd acd bcd 
max antichain: a bc bd 
dimension: 3

9.
id: 65301
bc -> d
a -> d
meet-irreducibles: b c abd acd bcd 
max antichain: abd acd bcd 
dimension: 3

10.
id: 63317
ab -> cd
a -> d
meet-irreducibles: bc ad bd acd bcd 
max antichain: bc ad bd 
dimension: 3

11.
id: 63255
bc -> d
ac -> d
ab -> cd
meet-irreducibles: a b c ad bd acd bcd 
max antichain: a b c 
dimension: 3

12.
id: 62775
ad -> c
bc -> d
ab -> cd
meet-irreducibles: a b ac bd acd bcd 
max antichain: a b 
dimension: 2

13.
id: 65297
b -> d
a -> d
meet-irreducibles: c abd acd bcd 
max antichain: abd acd bcd 
dimension: 3

14.
id: 63253
bc -> d
ab -> cd
a -> d
meet-irreducibles: b c ad bd acd bcd 
max antichain: b c ad 
dimension: 3

15.
id: 62805
a -> cd
meet-irreducibles: bc bd acd bcd 
max antichain: bc bd acd 
dimension: 3

16.
id: 62773
bc -> d
ab -> cd
a -> c
meet-irreducibles: b ac bd acd bcd 
max antichain: b ac 
dimension: 2

17.
id: 63239
c -> d
ab -> cd
meet-irreducibles: a b ad bd acd bcd 
max antichain: a b 
dimension: 2

18.
id: 65281
c -> d
b -> d
a -> d
meet-irreducibles: abd acd bcd 
max antichain: abd acd bcd 
dimension: 3

19.
id: 63249
ab -> cd
b -> d
a -> d
meet-irreducibles: c ad bd acd bcd 
max antichain: c ad bd 
dimension: 3

20.
id: 62769
ab -> cd
b -> d
a -> c
meet-irreducibles: ac bd acd bcd 
max antichain: ac bd 
dimension: 2

21.
id: 63237
c -> d
ab -> cd
a -> d
meet-irreducibles: b ad bd acd bcd 
max antichain: b ad 
dimension: 2

22.
id: 62741
bc -> d
a -> cd
meet-irreducibles: b c bd acd bcd 
max antichain: b c 
dimension: 2

23.
id: 54613
a -> bcd
meet-irreducibles: bc bd cd bcd 
max antichain: bc bd cd 
dimension: 3

24.
id: 63233
c -> d
ab -> cd
b -> d
a -> d
meet-irreducibles: ad bd acd bcd 
max antichain: ad bd 
dimension: 2

25.
id: 62737
b -> d
a -> cd
meet-irreducibles: c bd acd bcd 
max antichain: c bd 
dimension: 2

26.
id: 62725
c -> d
a -> cd
meet-irreducibles: b bd acd bcd 
max antichain: b acd 
dimension: 2

27.
id: 54549
bc -> d
a -> bcd
meet-irreducibles: b c bd cd bcd 
max antichain: b c 
dimension: 2

28.
id: 62721
c -> d
b -> d
a -> cd
meet-irreducibles: bd acd bcd 
max antichain: bd acd 
dimension: 2

29.
id: 61713
b -> cd
a -> cd
meet-irreducibles: c d acd bcd 
max antichain: c d 
dimension: 2

30.
id: 54545
b -> d
a -> bcd
meet-irreducibles: c bd cd bcd 
max antichain: c bd 
dimension: 2

31.
id: 61697
c -> d
b -> cd
a -> cd
meet-irreducibles: d acd bcd 
max antichain: acd bcd 
dimension: 2

32.
id: 54529
c -> d
b -> d
a -> bcd
meet-irreducibles: bd cd bcd 
max antichain: bd cd 
dimension: 2

33.
id: 53521
b -> cd
a -> bcd
meet-irreducibles: c d cd bcd 
max antichain: c d 
dimension: 2

34.
id: 53505
c -> d
b -> cd
a -> bcd
meet-irreducibles: d cd bcd 
max antichain: d 
dimension: 1

\end{verbatim} 

\subsection{Description of size 5 geometries} \chead{\sc Appendix A.2: Description of size 5 geometries} \begin{verbatim}
    
1.
id: 4294967295
meet-irreducibles: abcd abce abde acde bcde 
max antichain: abcd abce abde acde bcde 
dimension: 5

2.
id: 4294934527
abcd -> e
meet-irreducibles: abc abd acd bcd abce abde acde bcde 
max antichain: abc abd acd bcd 
dimension: 4

3.
id: 4294934399
abc -> e
meet-irreducibles: abd acd bcd abce abde acde bcde 
max antichain: abd acd bcd abce 
dimension: 4

4.
id: 4286545791
abc -> de
meet-irreducibles: abd acd bcd abe ace bce abde acde bcde 
max antichain: abd acd bcd abe ace bce 
dimension: 6

5.
id: 4294932351
abd -> e
abc -> e
meet-irreducibles: ab acd bcd abce abde acde bcde 
max antichain: acd bcd abce abde 
dimension: 4

6.
id: 4294932343
ab -> e
meet-irreducibles: acd bcd abce abde acde bcde 
max antichain: acd bcd abce abde 
dimension: 4

7.
id: 4294924159
acd -> e
abd -> e
abc -> e
meet-irreducibles: ab ac ad bcd abce abde acde bcde 
max antichain: ab ac ad bcd 
dimension: 4

8.
id: 4286543743
abd -> e
abc -> de
meet-irreducibles: ab acd bcd abe ace bce abde acde bcde 
max antichain: ab acd bcd ace bce 
dimension: 5

9.
id: 4294924151
acd -> e
ab -> e
meet-irreducibles: ac ad bcd abce abde acde bcde 
max antichain: bcd abce abde acde 
dimension: 4

10.
id: 4286543735
abc -> de
ab -> e
meet-irreducibles: acd bcd abe ace bce abde acde bcde 
max antichain: acd bcd abe ace bce 
dimension: 5

11.
id: 4294907775
bcd -> e
acd -> e
abd -> e
abc -> e
meet-irreducibles: ab ac bc ad bd cd abce abde acde bcde 
max antichain: ab ac bc ad bd cd 
dimension: 6

12.
id: 4286535551
acd -> e
abd -> e
abc -> de
meet-irreducibles: ab ac ad bcd abe ace bce abde acde bcde 
max antichain: ab ac ad bcd bce 
dimension: 5

13.
id: 4286013311
abe -> d
acd -> e
abc -> de
meet-irreducibles: ab ac abd bcd ace bce abde acde bcde 
max antichain: ab ac bcd bce 
dimension: 4

14.
id: 4294907767
bcd -> e
acd -> e
ab -> e
meet-irreducibles: ac bc ad bd cd abce abde acde bcde 
max antichain: ac bc ad bd cd 
dimension: 5

15.
id: 4286535543
acd -> e
abc -> de
ab -> e
meet-irreducibles: ac ad bcd abe ace bce abde acde bcde 
max antichain: ac ad bcd abe bce 
dimension: 5

16.
id: 4286535039
ad -> e
abc -> de
meet-irreducibles: ab ac bcd abe ace bce abde acde bcde 
max antichain: ab ac bcd bce 
dimension: 4

17.
id: 4286019447
ab -> de
meet-irreducibles: acd bcd ace bce abde acde bcde 
max antichain: acd bcd ace bce abde 
dimension: 5

18.
id: 4286519167
bcd -> e
acd -> e
abd -> e
abc -> de
meet-irreducibles: ab ac bc ad bd cd abe ace bce abde acde bcde 
max antichain: ab ac bc ad bd cd 
dimension: 6

19.
id: 4294924119
ac -> e
ab -> e
meet-irreducibles: ad bcd abce abde acde bcde 
max antichain: bcd abce abde acde 
dimension: 4

20.
id: 4286013303
acd -> e
abc -> de
ab -> d
meet-irreducibles: ac abd bcd ace bce abde acde bcde 
max antichain: ac abd bcd bce 
dimension: 4

21.
id: 4285996927
abe -> d
bcd -> e
acd -> e
abc -> de
meet-irreducibles: ab ac bc abd cd ace bce abde acde bcde 
max antichain: ab ac bc cd 
dimension: 4

22.
id: 4286535031
ad -> e
abc -> de
ab -> e
meet-irreducibles: ac bcd abe ace bce abde acde bcde 
max antichain: ac bcd abe bce 
dimension: 4

23.
id: 4286519159
bcd -> e
acd -> e
abc -> de
ab -> e
meet-irreducibles: ac bc ad bd cd abe ace bce abde acde bcde 
max antichain: ac bc ad bd cd abe 
dimension: 6

24.
id: 4286013271
abc -> de
ac -> e
ab -> d
meet-irreducibles: abd bcd ace bce abde acde bcde 
max antichain: abd bcd ace bce 
dimension: 4

25.
id: 4286011255
acd -> e
ab -> de
meet-irreducibles: ac ad bcd ace bce abde acde bcde 
max antichain: ac ad bcd bce 
dimension: 4

26.
id: 4286518655
bcd -> e
ad -> e
abc -> de
meet-irreducibles: ab ac bc bd cd abe ace bce abde acde bcde 
max antichain: ab ac bc bd cd 
dimension: 5

27.
id: 4285996895
abe -> d
bcd -> e
abc -> de
ac -> e
meet-irreducibles: ab bc abd cd ace bce abde acde bcde 
max antichain: ab bc cd ace 
dimension: 4

28.
id: 4151801719
ab -> cde
meet-irreducibles: acd bcd ace bce ade bde acde bcde 
max antichain: acd bcd ace bce ade bde 
dimension: 6

29.
id: 4294907199
ad -> e
bc -> e
meet-irreducibles: ab ac bd cd abce abde acde bcde 
max antichain: ab ac bd cd 
dimension: 4

30.
id: 4294923607
ad -> e
ac -> e
ab -> e
meet-irreducibles: a bcd abce abde acde bcde 
max antichain: bcd abce abde acde 
dimension: 4

31.
id: 4294907735
bcd -> e
ac -> e
ab -> e
meet-irreducibles: bc ad bd cd abce abde acde bcde 
max antichain: bc ad bd cd 
dimension: 4

32.
id: 4285992831
abe -> d
cd -> e
abc -> de
meet-irreducibles: ab ac bc abd ace bce abde acde bcde 
max antichain: ab ac bc 
dimension: 3

33.
id: 4286535511
abc -> de
ac -> e
ab -> e
meet-irreducibles: ad bcd abe ace bce abde acde bcde 
max antichain: ad bcd abe ace bce 
dimension: 5

34.
id: 4285996919
bcd -> e
acd -> e
abc -> de
ab -> d
meet-irreducibles: ac bc abd cd ace bce abde acde bcde 
max antichain: ac bc abd cd 
dimension: 4

35.
id: 4286518647
bcd -> e
ad -> e
abc -> de
ab -> e
meet-irreducibles: ac bc bd cd abe ace bce abde acde bcde 
max antichain: ac bc bd cd abe 
dimension: 5

36.
id: 4285992799
abe -> d
cd -> e
abc -> de
ac -> e
meet-irreducibles: ab bc abd ace bce abde acde bcde 
max antichain: ab bc ace 
dimension: 3

37.
id: 4286010743
ad -> e
ab -> de
meet-irreducibles: ac bcd ace bce abde acde bcde 
max antichain: ac bcd bce abde 
dimension: 4

38.
id: 4285996887
bcd -> e
abc -> de
ac -> e
ab -> d
meet-irreducibles: bc abd cd ace bce abde acde bcde 
max antichain: bc abd cd ace 
dimension: 4

39.
id: 4285994871
bcd -> e
acd -> e
ab -> de
meet-irreducibles: ac bc ad bd cd ace bce abde acde bcde 
max antichain: ac bc ad bd cd 
dimension: 5

40.
id: 4151793527
acd -> e
ab -> cde
meet-irreducibles: ac ad bcd ace bce ade bde acde bcde 
max antichain: ac ad bcd bce bde 
dimension: 5

41.
id: 4286517631
bd -> e
ad -> e
abc -> de
meet-irreducibles: ab ac bc cd abe ace bce abde acde bcde 
max antichain: ab ac bc cd 
dimension: 4

42.
id: 4294907191
ad -> e
bc -> e
ab -> e
meet-irreducibles: ac bd cd abce abde acde bcde 
max antichain: abce abde acde bcde 
dimension: 4

43.
id: 4283905911
ace -> d
bcd -> e
ab -> de
meet-irreducibles: ac bc bd acd ae bce abde acde bcde 
max antichain: ac bc bd ae 
dimension: 4

44.
id: 4285992823
cd -> e
abc -> de
ab -> d
meet-irreducibles: ac bc abd ace bce abde acde bcde 
max antichain: ac bc abd 
dimension: 3

45.
id: 4286518591
ad -> e
abc -> de
bc -> e
meet-irreducibles: ab ac bd cd abe ace bce abde acde bcde 
max antichain: ab ac bd cd bce 
dimension: 5

46.
id: 4294907223
bcd -> e
ad -> e
ac -> e
ab -> e
meet-irreducibles: a bc bd cd abce abde acde bcde 
max antichain: a bc bd cd 
dimension: 4

47.
id: 4285996831
abe -> d
abc -> de
bc -> e
ac -> e
meet-irreducibles: ab abd cd ace bce abde acde bcde 
max antichain: ab cd ace bce 
dimension: 4

48.
id: 4286534999
ad -> e
abc -> de
ac -> e
ab -> e
meet-irreducibles: a bcd abe ace bce abde acde bcde 
max antichain: bcd abe ace bce 
dimension: 4

49.
id: 4286519127
bcd -> e
abc -> de
ac -> e
ab -> e
meet-irreducibles: bc ad bd cd abe ace bce abde acde bcde 
max antichain: bc ad bd cd abe ace 
dimension: 6

50.
id: 4286011223
ac -> e
ab -> de
meet-irreducibles: ad bcd ace bce abde acde bcde 
max antichain: ad bcd ace bce 
dimension: 4

51.
id: 4294923605
a -> e
meet-irreducibles: bcd abce abde acde bcde 
max antichain: bcd abce abde acde 
dimension: 4

52.
id: 4294907671
bc -> e
ac -> e
ab -> e
meet-irreducibles: ad bd cd abce abde acde bcde 
max antichain: ad bd cd abce 
dimension: 4

53.
id: 4285994359
bcd -> e
ad -> e
ab -> de
meet-irreducibles: ac bc bd cd ace bce abde acde bcde 
max antichain: ac bc bd cd 
dimension: 4

54.
id: 4151777143
bcd -> e
acd -> e
ab -> cde
meet-irreducibles: ac bc ad bd cd ace bce ade bde acde bcde 
max antichain: ac bc ad bd cd 
dimension: 5

55.
id: 4283904887
ace -> d
bd -> e
ab -> de
meet-irreducibles: ac bc acd ae bce abde acde bcde 
max antichain: ac bc ae 
dimension: 3

56.
id: 4286518583
ad -> e
abc -> de
bc -> e
ab -> e
meet-irreducibles: ac bd cd abe ace bce abde acde bcde 
max antichain: ac bd cd abe bce 
dimension: 5

57.
id: 4149688183
ace -> d
bcd -> e
ab -> cde
meet-irreducibles: ac bc bd acd ae bce ade bde acde bcde 
max antichain: ac bc bd ae 
dimension: 4

58.
id: 4285992735
abe -> d
cd -> e
abc -> de
bc -> e
ac -> e
meet-irreducibles: ab c abd ace bce abde acde bcde 
max antichain: ab ace bce 
dimension: 3

59.
id: 4285996823
abc -> de
bc -> e
ac -> e
ab -> d
meet-irreducibles: abd cd ace bce abde acde bcde 
max antichain: abd cd ace bce 
dimension: 4

60.
id: 4286518615
bcd -> e
ad -> e
abc -> de
ac -> e
ab -> e
meet-irreducibles: a bc bd cd abe ace bce abde acde bcde 
max antichain: bc bd cd abe ace 
dimension: 5

61.
id: 4286010711
ad -> e
ac -> e
ab -> de
meet-irreducibles: a bcd ace bce abde acde bcde 
max antichain: bcd ace bce abde 
dimension: 4

62.
id: 4285994839
bcd -> e
ac -> e
ab -> de
meet-irreducibles: bc ad bd cd ace bce abde acde bcde 
max antichain: bc ad bd cd ace 
dimension: 5

63.
id: 4294907221
bcd -> e
a -> e
meet-irreducibles: bc bd cd abce abde acde bcde 
max antichain: abce abde acde bcde 
dimension: 4

64.
id: 4283914071
ac -> de
ab -> de
meet-irreducibles: ad bcd ae bce abde acde bcde 
max antichain: ad bcd ae bce 
dimension: 4

65.
id: 4151793495
ac -> e
ab -> cde
meet-irreducibles: ad bcd ace bce ade bde acde bcde 
max antichain: ad bcd ace bce bde 
dimension: 5

66.
id: 4286517599
bd -> e
ad -> e
abc -> de
ac -> e
meet-irreducibles: ab bc cd abe ace bce abde acde bcde 
max antichain: ab bc cd ace 
dimension: 4

67.
id: 4286534997
abc -> de
a -> e
meet-irreducibles: bcd abe ace bce abde acde bcde 
max antichain: bcd abe ace bce 
dimension: 4

68.
id: 4294907159
ad -> e
bc -> e
ac -> e
ab -> e
meet-irreducibles: a bd cd abce abde acde bcde 
max antichain: abce abde acde bcde 
dimension: 4

69.
id: 4283905879
bcd -> e
ac -> d
ab -> de
meet-irreducibles: bc bd acd ae bce abde acde bcde 
max antichain: bc bd acd ae 
dimension: 4

70.
id: 4286519063
abc -> de
bc -> e
ac -> e
ab -> e
meet-irreducibles: ad bd cd abe ace bce abde acde bcde 
max antichain: ad bd cd abe ace bce 
dimension: 6

71.
id: 4285992791
cd -> e
abc -> de
ac -> e
ab -> d
meet-irreducibles: bc abd ace bce abde acde bcde 
max antichain: bc abd ace 
dimension: 3

72.
id: 4285990775
cd -> e
ab -> de
meet-irreducibles: ac bc ad bd ace bce abde acde bcde 
max antichain: ac bc ad bd 
dimension: 4

73.
id: 4294906143
bd -> e
ad -> e
bc -> e
ac -> e
meet-irreducibles: ab cd abce abde acde bcde 
max antichain: abce abde acde bcde 
dimension: 4

74.
id: 4286513535
cd -> e
bd -> e
ad -> e
abc -> de
meet-irreducibles: ab ac bc d abe ace bce abde acde bcde 
max antichain: ab ac bc d 
dimension: 4

75.
id: 4286517623
bd -> e
ad -> e
abc -> de
ab -> e
meet-irreducibles: ac bc cd abe ace bce abde acde bcde 
max antichain: ac bc cd abe 
dimension: 4

76.
id: 4285992727
cd -> e
abc -> de
bc -> e
ac -> e
ab -> d
meet-irreducibles: c abd ace bce abde acde bcde 
max antichain: abd ace bce 
dimension: 3

77.
id: 4285992719
abe -> d
abc -> de
c -> e
meet-irreducibles: ab abd ace bce abde acde bcde 
max antichain: ab ace bce 
dimension: 3

78.
id: 4283773815
ae -> d
bd -> e
ab -> de
meet-irreducibles: ac bc acd bce abde acde bcde 
max antichain: ac bc abde 
dimension: 3

79.
id: 4149687159
ace -> d
bd -> e
ab -> cde
meet-irreducibles: ac bc acd ae bce ade bde acde bcde 
max antichain: ac bc ae bde 
dimension: 4

80.
id: 4285994295
ad -> e
bc -> e
ab -> de
meet-irreducibles: ac bd cd ace bce abde acde bcde 
max antichain: ac bd cd bce 
dimension: 4

81.
id: 4285990743
cd -> e
ac -> e
ab -> de
meet-irreducibles: bc ad bd ace bce abde acde bcde 
max antichain: bc ad bd ace 
dimension: 4

82.
id: 4286517591
bd -> e
ad -> e
abc -> de
ac -> e
ab -> e
meet-irreducibles: a bc cd abe ace bce abde acde bcde 
max antichain: bc cd abe ace 
dimension: 4

83.
id: 4283904823
ace -> d
bd -> e
bc -> e
ab -> de
meet-irreducibles: b ac acd ae bce abde acde bcde 
max antichain: b ac ae 
dimension: 3

84.
id: 4285994327
bcd -> e
ad -> e
ac -> e
ab -> de
meet-irreducibles: a bc bd cd ace bce abde acde bcde 
max antichain: a bc bd cd 
dimension: 4

85.
id: 4283913559
ad -> e
ac -> de
ab -> de
meet-irreducibles: a bcd ae bce abde acde bcde 
max antichain: bcd bce abde acde 
dimension: 4

86.
id: 4151792983
ad -> e
ac -> e
ab -> cde
meet-irreducibles: a bcd ace bce ade bde acde bcde 
max antichain: bcd ace bce ade bde 
dimension: 5

87.
id: 4283897687
bcd -> e
ac -> de
ab -> de
meet-irreducibles: bc ad bd cd ae bce abde acde bcde 
max antichain: bc ad bd cd ae 
dimension: 5

88.
id: 4151777111
bcd -> e
ac -> e
ab -> cde
meet-irreducibles: bc ad bd cd ace bce ade bde acde bcde 
max antichain: bc ad bd cd ace 
dimension: 5

89.
id: 4286518613
bcd -> e
abc -> de
a -> e
meet-irreducibles: bc bd cd abe ace bce abde acde bcde 
max antichain: bc bd cd abe ace 
dimension: 5

90.
id: 4149696343
ac -> de
ab -> cde
meet-irreducibles: ad bcd ae bce ade bde acde bcde 
max antichain: ad bcd ae bce bde 
dimension: 5

91.
id: 4294906135
bd -> e
ad -> e
bc -> e
ac -> e
ab -> e
meet-irreducibles: a b cd abce abde acde bcde 
max antichain: abce abde acde bcde 
dimension: 4

92.
id: 4283904855
bd -> e
ac -> d
ab -> de
meet-irreducibles: bc acd ae bce abde acde bcde 
max antichain: bc acd ae 
dimension: 3

93.
id: 4286010709
ab -> de
a -> e
meet-irreducibles: bcd ace bce abde acde bcde 
max antichain: bcd ace bce abde 
dimension: 4

94.
id: 4286518551
ad -> e
abc -> de
bc -> e
ac -> e
ab -> e
meet-irreducibles: a bd cd abe ace bce abde acde bcde 
max antichain: bd cd abe ace bce 
dimension: 5

95.
id: 4149688151
bcd -> e
ac -> d
ab -> cde
meet-irreducibles: bc bd acd ae bce ade bde acde bcde 
max antichain: bc bd acd ae 
dimension: 4

96.
id: 4286513527
cd -> e
bd -> e
ad -> e
abc -> de
ab -> e
meet-irreducibles: ac bc d abe ace bce abde acde bcde 
max antichain: ac bc d abe 
dimension: 4

97.
id: 4286513279
d -> e
abc -> de
meet-irreducibles: ab ac bc abe ace bce abde acde bcde 
max antichain: ab ac bc 
dimension: 3

98.
id: 4285990263
cd -> e
ad -> e
ab -> de
meet-irreducibles: ac bc bd ace bce abde acde bcde 
max antichain: ac bc bd 
dimension: 3

99.
id: 4285994775
bc -> e
ac -> e
ab -> de
meet-irreducibles: ad bd cd ace bce abde acde bcde 
max antichain: ad bd cd ace bce 
dimension: 5

100.
id: 4294907157
bc -> e
a -> e
meet-irreducibles: bd cd abce abde acde bcde 
max antichain: abce abde acde bcde 
dimension: 4

101.
id: 4151773047
cd -> e
ab -> cde
meet-irreducibles: ac bc ad bd ace bce ade bde acde bcde 
max antichain: ac bc ad bd 
dimension: 4

102.
id: 4286517535
bd -> e
ad -> e
abc -> de
bc -> e
ac -> e
meet-irreducibles: ab cd abe ace bce abde acde bcde 
max antichain: ab cd ace bce 
dimension: 4

103.
id: 4283905815
bc -> e
ac -> d
ab -> de
meet-irreducibles: bd acd ae bce abde acde bcde 
max antichain: bd acd ae bce 
dimension: 4

104.
id: 4285993335
bd -> e
ad -> e
ab -> de
meet-irreducibles: ac bc cd ace bce abde acde bcde 
max antichain: ac bc cd abde 
dimension: 4

105.
id: 4286513495
cd -> e
bd -> e
ad -> e
abc -> de
ac -> e
ab -> e
meet-irreducibles: a bc d abe ace bce abde acde bcde 
max antichain: bc d abe ace 
dimension: 4

106.
id: 4149556087
ae -> d
bd -> e
ab -> cde
meet-irreducibles: ac bc acd bce ade bde acde bcde 
max antichain: ac bc ade bde 
dimension: 4

107.
id: 4285990231
cd -> e
ad -> e
ac -> e
ab -> de
meet-irreducibles: a bc bd ace bce abde acde bcde 
max antichain: a bc bd 
dimension: 3

108.
id: 4116002679
ae -> cd
bcd -> e
ab -> cde
meet-irreducibles: ac bc ad bd acd bce bde acde bcde 
max antichain: ac bc ad bd 
dimension: 4

109.
id: 4151776567
ad -> e
bc -> e
ab -> cde
meet-irreducibles: ac bd cd ace bce ade bde acde bcde 
max antichain: ac bd cd bce ade 
dimension: 5

110.
id: 4151773015
cd -> e
ac -> e
ab -> cde
meet-irreducibles: bc ad bd ace bce ade bde acde bcde 
max antichain: bc ad bd ace 
dimension: 4

111.
id: 4294902039
cd -> e
bd -> e
ad -> e
bc -> e
ac -> e
ab -> e
meet-irreducibles: a b c d abce abde acde bcde 
max antichain: a b c d 
dimension: 4

112.
id: 4285993303
bd -> e
ad -> e
ac -> e
ab -> de
meet-irreducibles: a bc cd ace bce abde acde bcde 
max antichain: bc cd ace abde 
dimension: 4

113.
id: 4149687095
ace -> d
bd -> e
bc -> e
ab -> cde
meet-irreducibles: b ac acd ae bce ade bde acde bcde 
max antichain: ac ae bce bde 
dimension: 4

114.
id: 4283897175
bcd -> e
ad -> e
ac -> de
ab -> de
meet-irreducibles: a bc bd cd ae bce abde acde bcde 
max antichain: a bc bd cd 
dimension: 4

115.
id: 4151776599
bcd -> e
ad -> e
ac -> e
ab -> cde
meet-irreducibles: a bc bd cd ace bce ade bde acde bcde 
max antichain: bc bd cd ace ade 
dimension: 5

116.
id: 4283896663
bd -> e
ac -> de
ab -> de
meet-irreducibles: bc ad cd ae bce abde acde bcde 
max antichain: bc ad cd ae 
dimension: 4

117.
id: 4286513271
d -> e
abc -> de
ab -> e
meet-irreducibles: ac bc abe ace bce abde acde bcde 
max antichain: ac bc abe 
dimension: 3

118.
id: 4285990679
cd -> e
bc -> e
ac -> e
ab -> de
meet-irreducibles: c ad bd ace bce abde acde bcde 
max antichain: ad bd ace bce 
dimension: 4

119.
id: 4149695831
ad -> e
ac -> de
ab -> cde
meet-irreducibles: a bcd ae bce ade bde acde bcde 
max antichain: a bcd bce bde 
dimension: 4

120.
id: 4286517589
bd -> e
abc -> de
a -> e
meet-irreducibles: bc cd abe ace bce abde acde bcde 
max antichain: bc cd abe ace 
dimension: 4

121.
id: 4283766615
ae -> d
bcd -> e
ac -> de
ab -> de
meet-irreducibles: a bc ad bd cd bce abde acde bcde 
max antichain: a bc bd cd 
dimension: 4

122.
id: 4149679959
bcd -> e
ac -> de
ab -> cde
meet-irreducibles: bc ad bd cd ae bce ade bde acde bcde 
max antichain: bc ad bd cd ae 
dimension: 5

123.
id: 4285994325
bcd -> e
ab -> de
a -> e
meet-irreducibles: bc bd cd ace bce abde acde bcde 
max antichain: bc bd cd ace 
dimension: 4

124.
id: 4283773783
ae -> d
bd -> e
ac -> d
ab -> de
meet-irreducibles: a bc acd bce abde acde bcde 
max antichain: bc acd abde 
dimension: 3

125.
id: 4286517527
bd -> e
ad -> e
abc -> de
bc -> e
ac -> e
ab -> e
meet-irreducibles: a b cd abe ace bce abde acde bcde 
max antichain: cd abe ace bce 
dimension: 4

126.
id: 4149687127
bd -> e
ac -> d
ab -> cde
meet-irreducibles: bc acd ae bce ade bde acde bcde 
max antichain: bc acd ae bde 
dimension: 4

127.
id: 4283913557
ac -> de
ab -> de
a -> e
meet-irreducibles: bcd ae bce abde acde bcde 
max antichain: bcd bce abde acde 
dimension: 4

128.
id: 4151792981
ab -> cde
a -> e
meet-irreducibles: bcd ace bce ade bde acde bcde 
max antichain: bcd ace bce ade bde 
dimension: 5

129.
id: 4285994263
ad -> e
bc -> e
ac -> e
ab -> de
meet-irreducibles: a bd cd ace bce abde acde bcde 
max antichain: a bd cd bce 
dimension: 4

130.
id: 4285989239
cd -> e
bd -> e
ad -> e
ab -> de
meet-irreducibles: ac bc d ace bce abde acde bcde 
max antichain: ac bc d 
dimension: 3

131.
id: 4294906133
bd -> e
bc -> e
a -> e
meet-irreducibles: b cd abce abde acde bcde 
max antichain: abce abde acde bcde 
dimension: 4

132.
id: 4283897623
bc -> e
ac -> de
ab -> de
meet-irreducibles: ad bd cd ae bce abde acde bcde 
max antichain: ad bd cd ae bce 
dimension: 5

133.
id: 4151777047
bc -> e
ac -> e
ab -> cde
meet-irreducibles: ad bd cd ace bce ade bde acde bcde 
max antichain: ad bd cd ace bce 
dimension: 5

134.
id: 4286518549
abc -> de
bc -> e
a -> e
meet-irreducibles: bd cd abe ace bce abde acde bcde 
max antichain: bd cd abe ace bce 
dimension: 5

135.
id: 4283904791
bd -> e
bc -> e
ac -> d
ab -> de
meet-irreducibles: b acd ae bce abde acde bcde 
max antichain: b acd ae 
dimension: 3

136.
id: 4149688087
bc -> e
ac -> d
ab -> cde
meet-irreducibles: bd acd ae bce ade bde acde bcde 
max antichain: bd acd ae bce 
dimension: 4

137.
id: 4285992711
abc -> de
c -> e
ab -> d
meet-irreducibles: abd ace bce abde acde bcde 
max antichain: abd ace bce 
dimension: 3

138.
id: 4285990199
cd -> e
ad -> e
bc -> e
ab -> de
meet-irreducibles: ac bd ace bce abde acde bcde 
max antichain: ac bd bce 
dimension: 3

139.
id: 4283904819
ace -> d
ab -> de
b -> e
meet-irreducibles: ac acd ae bce abde acde bcde 
max antichain: ac ae bce 
dimension: 3

140.
id: 4285989207
cd -> e
bd -> e
ad -> e
ac -> e
ab -> de
meet-irreducibles: a bc d ace bce abde acde bcde 
max antichain: a bc d 
dimension: 3

141.
id: 4151772503
cd -> e
ad -> e
ac -> e
ab -> cde
meet-irreducibles: a bc bd ace bce ade bde acde bcde 
max antichain: bc bd ace ade 
dimension: 4

142.
id: 4283892567
cd -> e
bd -> e
ac -> de
ab -> de
meet-irreducibles: bc ad ae bce abde acde bcde 
max antichain: bc ad ae 
dimension: 3

143.
id: 4286513493
cd -> e
bd -> e
abc -> de
a -> e
meet-irreducibles: bc d abe ace bce abde acde bcde 
max antichain: bc d abe ace 
dimension: 4

144.
id: 4149675863
cd -> e
ac -> de
ab -> cde
meet-irreducibles: bc ad bd ae bce ade bde acde bcde 
max antichain: bc ad bd ae 
dimension: 4

145.
id: 4285990229
cd -> e
ab -> de
a -> e
meet-irreducibles: bc bd ace bce abde acde bcde 
max antichain: bc bd ace 
dimension: 3

146.
id: 4294906129
b -> e
a -> e
meet-irreducibles: cd abce abde acde bcde 
max antichain: abce abde acde bcde 
dimension: 4

147.
id: 4286513431
cd -> e
bd -> e
ad -> e
abc -> de
bc -> e
ac -> e
ab -> e
meet-irreducibles: a b c d abe ace bce abde acde bcde 
max antichain: a b c d 
dimension: 4

148.
id: 4285990663
c -> e
ab -> de
meet-irreducibles: ad bd ace bce abde acde bcde 
max antichain: ad bd ace bce 
dimension: 4

149.
id: 4283896151
bd -> e
ad -> e
ac -> de
ab -> de
meet-irreducibles: a bc cd ae bce abde acde bcde 
max antichain: a bc cd 
dimension: 3

150.
id: 4285990167
cd -> e
ad -> e
bc -> e
ac -> e
ab -> de
meet-irreducibles: a c bd ace bce abde acde bcde 
max antichain: a c bd 
dimension: 3

151.
id: 4283904787
ac -> d
ab -> de
b -> e
meet-irreducibles: acd ae bce abde acde bcde 
max antichain: acd ae bce 
dimension: 3

152.
id: 4116002615
ae -> cd
bc -> e
ab -> cde
meet-irreducibles: ac ad bd acd bce bde acde bcde 
max antichain: ac ad bd bce 
dimension: 4

153.
id: 4149679447
bcd -> e
ad -> e
ac -> de
ab -> cde
meet-irreducibles: a bc bd cd ae bce ade bde acde bcde 
max antichain: a bc bd cd 
dimension: 4

154.
id: 4283765591
ae -> d
bd -> e
ac -> de
ab -> de
meet-irreducibles: a bc ad cd bce abde acde bcde 
max antichain: a bc cd 
dimension: 3

155.
id: 4294902037
cd -> e
bd -> e
bc -> e
a -> e
meet-irreducibles: b c d abce abde acde bcde 
max antichain: abce abde acde bcde 
dimension: 4

156.
id: 4149678935
bd -> e
ac -> de
ab -> cde
meet-irreducibles: bc ad cd ae bce ade bde acde bcde 
max antichain: bc ad cd ae bde 
dimension: 5

157.
id: 4285988983
d -> e
ab -> de
meet-irreducibles: ac bc ace bce abde acde bcde 
max antichain: ac bc abde 
dimension: 3

158.
id: 4151772951
cd -> e
bc -> e
ac -> e
ab -> cde
meet-irreducibles: c ad bd ace bce ade bde acde bcde 
max antichain: ad bd ace bce 
dimension: 4

159.
id: 4285993301
bd -> e
ab -> de
a -> e
meet-irreducibles: bc cd ace bce abde acde bcde 
max antichain: bc cd ace abde 
dimension: 4

160.
id: 4149548887
ae -> d
bcd -> e
ac -> de
ab -> cde
meet-irreducibles: a bc ad bd cd bce ade bde acde bcde 
max antichain: a bc bd cd 
dimension: 4

161.
id: 4283897173
bcd -> e
ac -> de
ab -> de
a -> e
meet-irreducibles: bc bd cd ae bce abde acde bcde 
max antichain: bc bd cd ae 
dimension: 4

162.
id: 4151776597
bcd -> e
ab -> cde
a -> e
meet-irreducibles: bc bd cd ace bce ade bde acde bcde 
max antichain: bc bd cd ace ade 
dimension: 5

163.
id: 4149556055
ae -> d
bd -> e
ac -> d
ab -> cde
meet-irreducibles: a bc acd bce ade bde acde bcde 
max antichain: bc acd ade bde 
dimension: 4

164.
id: 4285993239
bd -> e
ad -> e
bc -> e
ac -> e
ab -> de
meet-irreducibles: a b cd ace bce abde acde bcde 
max antichain: cd ace bce abde 
dimension: 4

165.
id: 4283782485
a -> de
meet-irreducibles: bcd bce abde acde bcde 
max antichain: bcd bce abde acde 
dimension: 4

166.
id: 4149695829
ac -> de
ab -> cde
a -> e
meet-irreducibles: bcd ae bce ade bde acde bcde 
max antichain: bcd ae bce bde 
dimension: 4

167.
id: 4283766613
bcd -> e
ac -> de
ab -> de
a -> d
meet-irreducibles: bc ad bd cd bce abde acde bcde 
max antichain: bc ad bd cd 
dimension: 4

168.
id: 4116002647
ae -> cd
bcd -> e
ac -> d
ab -> cde
meet-irreducibles: a bc ad bd acd bce bde acde bcde 
max antichain: a bc bd 
dimension: 3

169.
id: 4283897111
ad -> e
bc -> e
ac -> de
ab -> de
meet-irreducibles: a bd cd ae bce abde acde bcde 
max antichain: a bd cd bce 
dimension: 4

170.
id: 4151776535
ad -> e
bc -> e
ac -> e
ab -> cde
meet-irreducibles: a bd cd ace bce ade bde acde bcde 
max antichain: bd cd ace bce ade 
dimension: 5

171.
id: 4283896599
bd -> e
bc -> e
ac -> de
ab -> de
meet-irreducibles: b ad cd ae bce abde acde bcde 
max antichain: b ad cd ae 
dimension: 4

172.
id: 4283773781
bd -> e
ab -> de
a -> d
meet-irreducibles: bc acd bce abde acde bcde 
max antichain: bc acd abde 
dimension: 3

173.
id: 4286517525
bd -> e
abc -> de
bc -> e
a -> e
meet-irreducibles: b cd abe ace bce abde acde bcde 
max antichain: cd abe ace bce 
dimension: 4

174.
id: 4283766551
ae -> d
bc -> e
ac -> de
ab -> de
meet-irreducibles: a ad bd cd bce abde acde bcde 
max antichain: a bd cd bce 
dimension: 4

175.
id: 4279703319
bc -> de
ac -> de
ab -> de
meet-irreducibles: ad bd cd ae be ce abde acde bcde 
max antichain: ad bd cd ae be ce 
dimension: 6

176.
id: 4149679895
bc -> e
ac -> de
ab -> cde
meet-irreducibles: ad bd cd ae bce ade bde acde bcde 
max antichain: ad bd cd ae bce 
dimension: 5

177.
id: 4285994261
bc -> e
ab -> de
a -> e
meet-irreducibles: bd cd ace bce abde acde bcde 
max antichain: bd cd ace bce 
dimension: 4

178.
id: 4283773719
ae -> d
bd -> e
bc -> e
ac -> d
ab -> de
meet-irreducibles: a b acd bce abde acde bcde 
max antichain: acd bce abde 
dimension: 3

179.
id: 4149687063
bd -> e
bc -> e
ac -> d
ab -> cde
meet-irreducibles: b acd ae bce ade bde acde bcde 
max antichain: acd ae bce bde 
dimension: 4

180.
id: 4286513239
d -> e
abc -> de
ac -> e
ab -> e
meet-irreducibles: a bc abe ace bce abde acde bcde 
max antichain: bc abe ace 
dimension: 3

181.
id: 4151772471
cd -> e
ad -> e
bc -> e
ab -> cde
meet-irreducibles: ac bd ace bce ade bde acde bcde 
max antichain: ac bd bce ade 
dimension: 4

182.
id: 4149687091
ace -> d
ab -> cde
b -> e
meet-irreducibles: ac acd ae bce ade bde acde bcde 
max antichain: ac ae bce bde 
dimension: 4

183.
id: 4294902033
cd -> e
b -> e
a -> e
meet-irreducibles: c d abce abde acde bcde 
max antichain: abce abde acde bcde 
dimension: 4

184.
id: 4286513237
d -> e
abc -> de
a -> e
meet-irreducibles: bc abe ace bce abde acde bcde 
max antichain: bc abe ace 
dimension: 3

185.
id: 4283892055
cd -> e
bd -> e
ad -> e
ac -> de
ab -> de
meet-irreducibles: a bc d ae bce abde acde bcde 
max antichain: a bc d 
dimension: 3

186.
id: 4286513175
d -> e
abc -> de
bc -> e
ac -> e
ab -> e
meet-irreducibles: a b c abe ace bce abde acde bcde 
max antichain: a b c 
dimension: 3

187.
id: 4149675351
cd -> e
ad -> e
ac -> de
ab -> cde
meet-irreducibles: a bc bd ae bce ade bde acde bcde 
max antichain: a bc bd 
dimension: 3

188.
id: 4114027831
ad -> ce
bc -> de
ab -> cde
meet-irreducibles: ac bd cd ae be ace bde acde bcde 
max antichain: ac bd cd ae be 
dimension: 5

189.
id: 4283761495
ae -> d
cd -> e
bd -> e
ac -> de
ab -> de
meet-irreducibles: a bc ad bce abde acde bcde 
max antichain: bc abde acde 
dimension: 3

190.
id: 4149674839
cd -> e
bd -> e
ac -> de
ab -> cde
meet-irreducibles: bc ad ae bce ade bde acde bcde 
max antichain: bc ad ae bde 
dimension: 4

191.
id: 4285989205
cd -> e
bd -> e
ab -> de
a -> e
meet-irreducibles: bc d ace bce abde acde bcde 
max antichain: bc d ace 
dimension: 3

192.
id: 4149544791
ae -> d
cd -> e
ac -> de
ab -> cde
meet-irreducibles: a bc ad bd bce ade bde acde bcde 
max antichain: a bc bd 
dimension: 3

193.
id: 4283896595
ac -> de
ab -> de
b -> e
meet-irreducibles: ad cd ae bce abde acde bcde 
max antichain: ad cd ae bce 
dimension: 4

194.
id: 4285990151
ad -> e
c -> e
ab -> de
meet-irreducibles: a bd ace bce abde acde bcde 
max antichain: a bd bce 
dimension: 3

195.
id: 4151772501
cd -> e
ab -> cde
a -> e
meet-irreducibles: bc bd ace bce ade bde acde bcde 
max antichain: bc bd ace ade 
dimension: 4

196.
id: 4286517521
abc -> de
b -> e
a -> e
meet-irreducibles: cd abe ace bce abde acde bcde 
max antichain: cd abe ace bce 
dimension: 4

197.
id: 4116002103
ae -> cd
ad -> c
bc -> e
ab -> cde
meet-irreducibles: a ac bd acd bce bde acde bcde 
max antichain: a bd bce 
dimension: 3

198.
id: 4285989143
cd -> e
bd -> e
ad -> e
bc -> e
ac -> e
ab -> de
meet-irreducibles: a b c d ace bce abde acde bcde 
max antichain: a b c d 
dimension: 4

199.
id: 4116001591
ae -> cd
bd -> e
bc -> e
ab -> cde
meet-irreducibles: b ac ad acd bce bde acde bcde 
max antichain: ac ad bce bde 
dimension: 4

200.
id: 4151772935
c -> e
ab -> cde
meet-irreducibles: ad bd ace bce ade bde acde bcde 
max antichain: ad bd ace bce 
dimension: 4

201.
id: 4149678423
bd -> e
ad -> e
ac -> de
ab -> cde
meet-irreducibles: a bc cd ae bce ade bde acde bcde 
max antichain: a bc cd bde 
dimension: 4

202.
id: 4151772439
cd -> e
ad -> e
bc -> e
ac -> e
ab -> cde
meet-irreducibles: a c bd ace bce ade bde acde bcde 
max antichain: bd ace bce ade 
dimension: 4

203.
id: 4149687059
ac -> d
ab -> cde
b -> e
meet-irreducibles: acd ae bce ade bde acde bcde 
max antichain: acd ae bce bde 
dimension: 4

204.
id: 4283892503
cd -> e
bd -> e
bc -> e
ac -> de
ab -> de
meet-irreducibles: b c ad ae bce abde acde bcde 
max antichain: b c ad ae 
dimension: 4

205.
id: 4149547863
ae -> d
bd -> e
ac -> de
ab -> cde
meet-irreducibles: a bc ad cd bce ade bde acde bcde 
max antichain: a bc cd bde 
dimension: 4

206.
id: 4286513429
cd -> e
bd -> e
abc -> de
bc -> e
a -> e
meet-irreducibles: b c d abe ace bce abde acde bcde 
max antichain: d abe ace bce 
dimension: 4

207.
id: 4149675799
cd -> e
bc -> e
ac -> de
ab -> cde
meet-irreducibles: c ad bd ae bce ade bde acde bcde 
max antichain: c ad bd ae 
dimension: 4

208.
id: 4283896149
bd -> e
ac -> de
ab -> de
a -> e
meet-irreducibles: bc cd ae bce abde acde bcde 
max antichain: bc cd ae 
dimension: 3

209.
id: 4285990165
cd -> e
bc -> e
ab -> de
a -> e
meet-irreducibles: c bd ace bce abde acde bcde 
max antichain: bd ace bce 
dimension: 3

210.
id: 4283766101
bcd -> e
a -> de
meet-irreducibles: bc bd cd bce abde acde bcde 
max antichain: bc bd cd 
dimension: 3

211.
id: 4149679445
bcd -> e
ac -> de
ab -> cde
a -> e
meet-irreducibles: bc bd cd ae bce ade bde acde bcde 
max antichain: bc bd cd ae 
dimension: 4

212.
id: 4283765589
bd -> e
ac -> de
ab -> de
a -> d
meet-irreducibles: bc ad cd bce abde acde bcde 
max antichain: bc ad cd 
dimension: 3

213.
id: 4283896087
bd -> e
ad -> e
bc -> e
ac -> de
ab -> de
meet-irreducibles: a b cd ae bce abde acde bcde 
max antichain: a b cd 
dimension: 3

214.
id: 4151775511
bd -> e
ad -> e
bc -> e
ac -> e
ab -> cde
meet-irreducibles: a b cd ace bce ade bde acde bcde 
max antichain: cd ace bce ade bde 
dimension: 5

215.
id: 4149564757
ab -> cde
a -> de
meet-irreducibles: bcd bce ade bde acde bcde 
max antichain: bcd bce ade bde 
dimension: 4

216.
id: 4149548885
bcd -> e
ac -> de
ab -> cde
a -> d
meet-irreducibles: bc ad bd cd bce ade bde acde bcde 
max antichain: bc ad bd cd 
dimension: 4

217.
id: 4279702807
ad -> e
bc -> de
ac -> de
ab -> de
meet-irreducibles: a bd cd ae be ce abde acde bcde 
max antichain: a bd cd be ce 
dimension: 5

218.
id: 4149679383
ad -> e
bc -> e
ac -> de
ab -> cde
meet-irreducibles: a bd cd ae bce ade bde acde bcde 
max antichain: a bd cd bce 
dimension: 4

219.
id: 4283765527
ae -> d
bd -> e
bc -> e
ac -> de
ab -> de
meet-irreducibles: a b ad cd bce abde acde bcde 
max antichain: a b cd 
dimension: 3

220.
id: 4149678871
bd -> e
bc -> e
ac -> de
ab -> cde
meet-irreducibles: b ad cd ae bce ade bde acde bcde 
max antichain: ad cd ae bce bde 
dimension: 5

221.
id: 4149556053
bd -> e
ab -> cde
a -> d
meet-irreducibles: bc acd bce ade bde acde bcde 
max antichain: bc acd ade bde 
dimension: 4

222.
id: 4285993237
bd -> e
bc -> e
ab -> de
a -> e
meet-irreducibles: b cd ace bce abde acde bcde 
max antichain: cd ace bce abde 
dimension: 4

223.
id: 4149548823
ae -> d
bc -> e
ac -> de
ab -> cde
meet-irreducibles: a ad bd cd bce ade bde acde bcde 
max antichain: a bd cd bce 
dimension: 4

224.
id: 4145485591
bc -> de
ac -> de
ab -> cde
meet-irreducibles: ad bd cd ae be ce ade bde acde bcde 
max antichain: ad bd cd ae be ce 
dimension: 6

225.
id: 4116002645
ae -> cd
bcd -> e
ab -> cde
a -> d
meet-irreducibles: bc ad bd acd bce bde acde bcde 
max antichain: bc ad bd 
dimension: 3

226.
id: 4283897109
bc -> e
ac -> de
ab -> de
a -> e
meet-irreducibles: bd cd ae bce abde acde bcde 
max antichain: bd cd ae bce 
dimension: 4

227.
id: 4151776533
bc -> e
ab -> cde
a -> e
meet-irreducibles: bd cd ace bce ade bde acde bcde 
max antichain: bd cd ace bce ade 
dimension: 5

228.
id: 4149555991
ae -> d
bd -> e
bc -> e
ac -> d
ab -> cde
meet-irreducibles: a b acd bce ade bde acde bcde 
max antichain: acd bce ade bde 
dimension: 4

229.
id: 4283766549
bc -> e
ac -> de
ab -> de
a -> d
meet-irreducibles: ad bd cd bce abde acde bcde 
max antichain: ad bd cd bce 
dimension: 4

230.
id: 4116002583
ae -> cd
bc -> e
ac -> d
ab -> cde
meet-irreducibles: a ad bd acd bce bde acde bcde 
max antichain: a bd bce 
dimension: 3

231.
id: 4285988951
d -> e
ac -> e
ab -> de
meet-irreducibles: a bc ace bce abde acde bcde 
max antichain: bc ace abde 
dimension: 3

232.
id: 4283773717
bd -> e
bc -> e
ab -> de
a -> d
meet-irreducibles: b acd bce abde acde bcde 
max antichain: acd bce abde 
dimension: 3

233.
id: 4116002581
ae -> cd
bc -> e
ab -> cde
a -> d
meet-irreducibles: ad bd acd bce bde acde bcde 
max antichain: ad bd bce 
dimension: 3

234.
id: 4283892499
cd -> e
ac -> de
ab -> de
b -> e
meet-irreducibles: c ad ae bce abde acde bcde 
max antichain: c ad ae 
dimension: 3

235.
id: 4286513425
cd -> e
abc -> de
b -> e
a -> e
meet-irreducibles: c d abe ace bce abde acde bcde 
max antichain: d abe ace bce 
dimension: 4

236.
id: 4285988949
d -> e
ab -> de
a -> e
meet-irreducibles: bc ace bce abde acde bcde 
max antichain: bc ace abde 
dimension: 3

237.
id: 4149674327
cd -> e
bd -> e
ad -> e
ac -> de
ab -> cde
meet-irreducibles: a bc d ae bce ade bde acde bcde 
max antichain: a bc d 
dimension: 3

238.
id: 4285988887
d -> e
bc -> e
ac -> e
ab -> de
meet-irreducibles: a b c ace bce abde acde bcde 
max antichain: a b c 
dimension: 3

239.
id: 4283896083
ad -> e
ac -> de
ab -> de
b -> e
meet-irreducibles: a cd ae bce abde acde bcde 
max antichain: a cd bce 
dimension: 3

240.
id: 4113896759
ae -> c
ad -> ce
bc -> de
ab -> cde
meet-irreducibles: a ac bd cd be ace bde acde bcde 
max antichain: a bd cd be 
dimension: 4

241.
id: 4149543767
ae -> d
cd -> e
bd -> e
ac -> de
ab -> cde
meet-irreducibles: a bc ad bce ade bde acde bcde 
max antichain: a bc bde 
dimension: 3

242.
id: 4285989127
bd -> e
ad -> e
c -> e
ab -> de
meet-irreducibles: a b d ace bce abde acde bcde 
max antichain: a b d 
dimension: 3

243.
id: 4283892053
cd -> e
bd -> e
ac -> de
ab -> de
a -> e
meet-irreducibles: bc d ae bce abde acde bcde 
max antichain: bc d ae 
dimension: 3

244.
id: 4283765523
ae -> d
ac -> de
ab -> de
b -> e
meet-irreducibles: a ad cd bce abde acde bcde 
max antichain: a cd bce 
dimension: 3

245.
id: 4149678867
ac -> de
ab -> cde
b -> e
meet-irreducibles: ad cd ae bce ade bde acde bcde 
max antichain: ad cd ae bce bde 
dimension: 5

246.
id: 4286513173
d -> e
abc -> de
bc -> e
a -> e
meet-irreducibles: b c abe ace bce abde acde bcde 
max antichain: abe ace bce 
dimension: 3

247.
id: 4151772423
ad -> e
c -> e
ab -> cde
meet-irreducibles: a bd ace bce ade bde acde bcde 
max antichain: bd ace bce ade 
dimension: 4

248.
id: 4149675349
cd -> e
ac -> de
ab -> cde
a -> e
meet-irreducibles: bc bd ae bce ade bde acde bcde 
max antichain: bc bd ae 
dimension: 3

249.
id: 4285993233
ab -> de
b -> e
a -> e
meet-irreducibles: cd ace bce abde acde bcde 
max antichain: cd ace bce abde 
dimension: 4

250.
id: 4283761493
cd -> e
bd -> e
ac -> de
ab -> de
a -> d
meet-irreducibles: bc ad bce abde acde bcde 
max antichain: bc abde acde 
dimension: 3

251.
id: 4283891991
cd -> e
bd -> e
ad -> e
bc -> e
ac -> de
ab -> de
meet-irreducibles: a b c d ae bce abde acde bcde 
max antichain: a b c d 
dimension: 4

252.
id: 4151771415
cd -> e
bd -> e
ad -> e
bc -> e
ac -> e
ab -> cde
meet-irreducibles: a b c d ace bce ade bde acde bcde 
max antichain: a b c d 
dimension: 4

253.
id: 4149675783
ac -> de
c -> e
ab -> cde
meet-irreducibles: ad bd ae bce ade bde acde bcde 
max antichain: ad bd ae bce 
dimension: 4

254.
id: 4149544789
cd -> e
ac -> de
ab -> cde
a -> d
meet-irreducibles: bc ad bd bce ade bde acde bcde 
max antichain: bc ad bd 
dimension: 3

255.
id: 4149675287
cd -> e
ad -> e
bc -> e
ac -> de
ab -> cde
meet-irreducibles: a c bd ae bce ade bde acde bcde 
max antichain: a c bd 
dimension: 3

256.
id: 4285990149
c -> e
ab -> de
a -> e
meet-irreducibles: bd ace bce abde acde bcde 
max antichain: bd ace bce 
dimension: 3

257.
id: 4283761431
ae -> d
cd -> e
bd -> e
bc -> e
ac -> de
ab -> de
meet-irreducibles: a b c ad bce abde acde bcde 
max antichain: a b c 
dimension: 3

258.
id: 4149674775
cd -> e
bd -> e
bc -> e
ac -> de
ab -> cde
meet-irreducibles: b c ad ae bce ade bde acde bcde 
max antichain: b c ad ae 
dimension: 4

259.
id: 4116002101
ae -> cd
bc -> e
ab -> cde
a -> c
meet-irreducibles: ac bd acd bce bde acde bcde 
max antichain: ac bd bce 
dimension: 3

260.
id: 4285989141
cd -> e
bd -> e
bc -> e
ab -> de
a -> e
meet-irreducibles: b c d ace bce abde acde bcde 
max antichain: b c d 
dimension: 3

261.
id: 4149544727
ae -> d
cd -> e
bc -> e
ac -> de
ab -> cde
meet-irreducibles: a c ad bd bce ade bde acde bcde 
max antichain: a c bd 
dimension: 3

262.
id: 4145481495
cd -> e
bc -> de
ac -> de
ab -> cde
meet-irreducibles: c ad bd ae be ce ade bde acde bcde 
max antichain: c ad bd ae be 
dimension: 5

263.
id: 4283765077
bd -> e
a -> de
meet-irreducibles: bc cd bce abde acde bcde 
max antichain: bc cd abde 
dimension: 3

264.
id: 4149678421
bd -> e
ac -> de
ab -> cde
a -> e
meet-irreducibles: bc cd ae bce ade bde acde bcde 
max antichain: bc cd ae bde 
dimension: 4

265.
id: 4283773713
ab -> de
b -> e
a -> d
meet-irreducibles: acd bce abde acde bcde 
max antichain: acd bce abde 
dimension: 3

266.
id: 4151772437
cd -> e
bc -> e
ab -> cde
a -> e
meet-irreducibles: c bd ace bce ade bde acde bcde 
max antichain: bd ace bce ade 
dimension: 4

267.
id: 4149548373
bcd -> e
ab -> cde
a -> de
meet-irreducibles: bc bd cd bce ade bde acde bcde 
max antichain: bc bd cd ade 
dimension: 4

268.
id: 4149547861
bd -> e
ac -> de
ab -> cde
a -> d
meet-irreducibles: bc ad cd bce ade bde acde bcde 
max antichain: bc ad cd bde 
dimension: 4

269.
id: 4279701783
bd -> e
ad -> e
bc -> de
ac -> de
ab -> de
meet-irreducibles: a b cd ae be ce abde acde bcde 
max antichain: a b cd ce 
dimension: 4

270.
id: 4149678359
bd -> e
ad -> e
bc -> e
ac -> de
ab -> cde
meet-irreducibles: a b cd ae bce ade bde acde bcde 
max antichain: a cd bce bde 
dimension: 4

271.
id: 4116010325
a -> cde
meet-irreducibles: bcd bce bde acde bcde 
max antichain: bcd bce bde acde 
dimension: 4

272.
id: 4145485079
ad -> e
bc -> de
ac -> de
ab -> cde
meet-irreducibles: a bd cd ae be ce ade bde acde bcde 
max antichain: a bd cd be ce 
dimension: 5

273.
id: 4114027799
ad -> ce
bc -> de
ac -> e
ab -> cde
meet-irreducibles: a bd cd ae be ace bde acde bcde 
max antichain: a bd cd be 
dimension: 4

274.
id: 4116002133
bcd -> e
ab -> cde
a -> cd
meet-irreducibles: bc bd acd bce bde acde bcde 
max antichain: bc bd acd 
dimension: 3

275.
id: 4279571223
ae -> d
bd -> e
bc -> de
ac -> de
ab -> de
meet-irreducibles: a b ad cd be ce abde acde bcde 
max antichain: a b cd ce 
dimension: 4

276.
id: 4149547799
ae -> d
bd -> e
bc -> e
ac -> de
ab -> cde
meet-irreducibles: a b ad cd bce ade bde acde bcde 
max antichain: a cd bce bde 
dimension: 4

277.
id: 4283896085
bd -> e
bc -> e
ac -> de
ab -> de
a -> e
meet-irreducibles: b cd ae bce abde acde bcde 
max antichain: b cd ae 
dimension: 3

278.
id: 4151775509
bd -> e
bc -> e
ab -> cde
a -> e
meet-irreducibles: b cd ace bce ade bde acde bcde 
max antichain: cd ace bce ade bde 
dimension: 5

279.
id: 4283766037
bc -> e
a -> de
meet-irreducibles: bd cd bce abde acde bcde 
max antichain: bd cd bce 
dimension: 3

280.
id: 4279702805
bc -> de
ac -> de
ab -> de
a -> e
meet-irreducibles: bd cd ae be ce abde acde bcde 
max antichain: bd cd ae be ce 
dimension: 5

281.
id: 4149679381
bc -> e
ac -> de
ab -> cde
a -> e
meet-irreducibles: bd cd ae bce ade bde acde bcde 
max antichain: bd cd ae bce 
dimension: 4

282.
id: 4283765525
bd -> e
bc -> e
ac -> de
ab -> de
a -> d
meet-irreducibles: b ad cd bce abde acde bcde 
max antichain: b ad cd 
dimension: 3

283.
id: 4116001559
ae -> cd
bd -> e
bc -> e
ac -> d
ab -> cde
meet-irreducibles: a b ad acd bce bde acde bcde 
max antichain: a bce bde 
dimension: 3

284.
id: 4149548821
bc -> e
ac -> de
ab -> cde
a -> d
meet-irreducibles: ad bd cd bce ade bde acde bcde 
max antichain: ad bd cd bce 
dimension: 4

285.
id: 4283891799
d -> e
ac -> de
ab -> de
meet-irreducibles: a bc ae bce abde acde bcde 
max antichain: bc abde acde 
dimension: 3

286.
id: 4294902017
c -> e
b -> e
a -> e
meet-irreducibles: d abce abde acde bcde 
max antichain: abce abde acde bcde 
dimension: 4

287.
id: 4116001587
ae -> cd
ab -> cde
b -> e
meet-irreducibles: ac ad acd bce bde acde bcde 
max antichain: ac ad bce bde 
dimension: 4

288.
id: 4114023735
cd -> e
ad -> ce
bc -> de
ab -> cde
meet-irreducibles: ac bd ae be ace bde acde bcde 
max antichain: ac bd ae be 
dimension: 4

289.
id: 4149555989
bd -> e
bc -> e
ab -> cde
a -> d
meet-irreducibles: b acd bce ade bde acde bcde 
max antichain: acd bce ade bde 
dimension: 4

290.
id: 4283761427
ae -> d
cd -> e
ac -> de
ab -> de
b -> e
meet-irreducibles: a c ad bce abde acde bcde 
max antichain: bce abde acde 
dimension: 3

291.
id: 4149674771
cd -> e
ac -> de
ab -> cde
b -> e
meet-irreducibles: c ad ae bce ade bde acde bcde 
max antichain: c ad ae bde 
dimension: 4

292.
id: 4285989137
cd -> e
ab -> de
b -> e
a -> e
meet-irreducibles: c d ace bce abde acde bcde 
max antichain: d ace bce 
dimension: 3

293.
id: 4285988871
d -> e
c -> e
ab -> de
meet-irreducibles: a b ace bce abde acde bcde 
max antichain: ace bce abde 
dimension: 3

294.
id: 4283891797
d -> e
ac -> de
ab -> de
a -> e
meet-irreducibles: bc ae bce abde acde bcde 
max antichain: bc abde acde 
dimension: 3

295.
id: 4283891735
d -> e
bc -> e
ac -> de
ab -> de
meet-irreducibles: a b c ae bce abde acde bcde 
max antichain: a b c 
dimension: 3

296.
id: 4113895735
ae -> c
bd -> e
ad -> ce
bc -> de
ab -> cde
meet-irreducibles: a b ac cd be ace bde acde bcde 
max antichain: a b cd 
dimension: 3

297.
id: 4149678355
ad -> e
ac -> de
ab -> cde
b -> e
meet-irreducibles: a cd ae bce ade bde acde bcde 
max antichain: a cd bce bde 
dimension: 4

298.
id: 4113634615
be -> d
ae -> c
ad -> ce
bc -> de
ab -> cde
meet-irreducibles: a b ac bd cd ace bde acde bcde 
max antichain: a b cd 
dimension: 3

299.
id: 4151771399
bd -> e
ad -> e
c -> e
ab -> cde
meet-irreducibles: a b d ace bce ade bde acde bcde 
max antichain: ace bce ade bde 
dimension: 4

300.
id: 4283760981
cd -> e
bd -> e
a -> de
meet-irreducibles: bc d bce abde acde bcde 
max antichain: bc abde acde 
dimension: 3

301.
id: 4149674325
cd -> e
bd -> e
ac -> de
ab -> cde
a -> e
meet-irreducibles: bc d ae bce ade bde acde bcde 
max antichain: bc d ae 
dimension: 3

302.
id: 4149547795
ae -> d
ac -> de
ab -> cde
b -> e
meet-irreducibles: a ad cd bce ade bde acde bcde 
max antichain: a cd bce bde 
dimension: 4

303.
id: 4285988885
d -> e
bc -> e
ab -> de
a -> e
meet-irreducibles: b c ace bce abde acde bcde 
max antichain: ace bce abde 
dimension: 3

304.
id: 4149675271
ad -> e
ac -> de
c -> e
ab -> cde
meet-irreducibles: a bd ae bce ade bde acde bcde 
max antichain: a bd bce 
dimension: 3

305.
id: 4149544277
cd -> e
ab -> cde
a -> de
meet-irreducibles: bc bd bce ade bde acde bcde 
max antichain: bc bd ade 
dimension: 3

306.
id: 4283896081
ac -> de
ab -> de
b -> e
a -> e
meet-irreducibles: cd ae bce abde acde bcde 
max antichain: cd ae bce 
dimension: 3

307.
id: 4151775505
ab -> cde
b -> e
a -> e
meet-irreducibles: cd ace bce ade bde acde bcde 
max antichain: cd ace bce ade bde 
dimension: 5

308.
id: 4113896757
ad -> ce
bc -> de
ab -> cde
a -> c
meet-irreducibles: ac bd cd be ace bde acde bcde 
max antichain: ac bd cd be 
dimension: 4

309.
id: 4149674759
bd -> e
ac -> de
c -> e
ab -> cde
meet-irreducibles: b ad ae bce ade bde acde bcde 
max antichain: ad ae bce bde 
dimension: 4

310.
id: 4149543765
cd -> e
bd -> e
ac -> de
ab -> cde
a -> d
meet-irreducibles: bc ad bce ade bde acde bcde 
max antichain: bc ad bde 
dimension: 3

311.
id: 4279697687
cd -> e
bd -> e
ad -> e
bc -> de
ac -> de
ab -> de
meet-irreducibles: a b c d ae be ce abde acde bcde 
max antichain: a b c d 
dimension: 4

312.
id: 4149674263
cd -> e
bd -> e
ad -> e
bc -> e
ac -> de
ab -> cde
meet-irreducibles: a b c d ae bce ade bde acde bcde 
max antichain: a b c d 
dimension: 4

313.
id: 4285989125
bd -> e
c -> e
ab -> de
a -> e
meet-irreducibles: b d ace bce abde acde bcde 
max antichain: b d ace 
dimension: 3

314.
id: 4149544711
ae -> d
ac -> de
c -> e
ab -> cde
meet-irreducibles: a ad bd bce ade bde acde bcde 
max antichain: a bd bce 
dimension: 3

315.
id: 4145481479
bc -> de
ac -> de
c -> e
ab -> cde
meet-irreducibles: ad bd ae be ce ade bde acde bcde 
max antichain: ad bd ae be ce 
dimension: 5

316.
id: 4283765521
ac -> de
ab -> de
b -> e
a -> d
meet-irreducibles: ad cd bce abde acde bcde 
max antichain: ad cd bce 
dimension: 3

317.
id: 4116001555
ae -> cd
ac -> d
ab -> cde
b -> e
meet-irreducibles: a ad acd bce bde acde bcde 
max antichain: a bce bde 
dimension: 3

318.
id: 4145480983
cd -> e
ad -> e
bc -> de
ac -> de
ab -> cde
meet-irreducibles: a c bd ae be ce ade bde acde bcde 
max antichain: a c bd be 
dimension: 4

319.
id: 4114023703
cd -> e
ad -> ce
bc -> de
ac -> e
ab -> cde
meet-irreducibles: a c bd ae be ace bde acde bcde 
max antichain: a c bd be 
dimension: 4

320.
id: 4151772421
c -> e
ab -> cde
a -> e
meet-irreducibles: bd ace bce ade bde acde bcde 
max antichain: bd ace bce ade 
dimension: 4

321.
id: 4279567127
ae -> d
cd -> e
bd -> e
bc -> de
ac -> de
ab -> de
meet-irreducibles: a b c ad be ce abde acde bcde 
max antichain: a b c 
dimension: 3

322.
id: 4149543703
ae -> d
cd -> e
bd -> e
bc -> e
ac -> de
ab -> cde
meet-irreducibles: a b c ad bce ade bde acde bcde 
max antichain: a b c 
dimension: 3

323.
id: 4283891989
cd -> e
bd -> e
bc -> e
ac -> de
ab -> de
a -> e
meet-irreducibles: b c d ae bce abde acde bcde 
max antichain: b c d ae 
dimension: 4

324.
id: 4151771413
cd -> e
bd -> e
bc -> e
ab -> cde
a -> e
meet-irreducibles: b c d ace bce ade bde acde bcde 
max antichain: ace bce ade bde 
dimension: 4

325.
id: 4145350423
ae -> d
cd -> e
bc -> de
ac -> de
ab -> cde
meet-irreducibles: a c ad bd be ce ade bde acde bcde 
max antichain: a c bd be 
dimension: 4

326.
id: 4149547349
bd -> e
ab -> cde
a -> de
meet-irreducibles: bc cd bce ade bde acde bcde 
max antichain: bc cd ade bde 
dimension: 4

327.
id: 4149555985
ab -> cde
b -> e
a -> d
meet-irreducibles: acd bce ade bde acde bcde 
max antichain: acd bce ade bde 
dimension: 4

328.
id: 4149675285
cd -> e
bc -> e
ac -> de
ab -> cde
a -> e
meet-irreducibles: c bd ae bce ade bde acde bcde 
max antichain: c bd ae 
dimension: 3

329.
id: 4115993941
bcd -> e
a -> cde
meet-irreducibles: bc bd cd bce bde acde bcde 
max antichain: bc bd cd 
dimension: 3

330.
id: 4283761429
cd -> e
bd -> e
bc -> e
ac -> de
ab -> de
a -> d
meet-irreducibles: b c ad bce abde acde bcde 
max antichain: b c ad 
dimension: 3

331.
id: 4145484055
bd -> e
ad -> e
bc -> de
ac -> de
ab -> cde
meet-irreducibles: a b cd ae be ce ade bde acde bcde 
max antichain: a b cd ce 
dimension: 4

332.
id: 4114026775
bd -> e
ad -> ce
bc -> de
ac -> e
ab -> cde
meet-irreducibles: a b cd ae be ace bde acde bcde 
max antichain: a b cd 
dimension: 3

333.
id: 4149544725
cd -> e
bc -> e
ac -> de
ab -> cde
a -> d
meet-irreducibles: c ad bd bce ade bde acde bcde 
max antichain: c ad bd 
dimension: 3

334.
id: 3579139413
a -> bcde
meet-irreducibles: bcd bce bde cde bcde 
max antichain: bcd bce bde cde 
dimension: 4

335.
id: 4145353495
ae -> d
bd -> e
bc -> de
ac -> de
ab -> cde
meet-irreducibles: a b ad cd be ce ade bde acde bcde 
max antichain: a b cd ce 
dimension: 4

336.
id: 4283765013
bd -> e
bc -> e
a -> de
meet-irreducibles: b cd bce abde acde bcde 
max antichain: cd bce abde 
dimension: 3

337.
id: 4279701781
bd -> e
bc -> de
ac -> de
ab -> de
a -> e
meet-irreducibles: b cd ae be ce abde acde bcde 
max antichain: b cd ae ce 
dimension: 4

338.
id: 4149678357
bd -> e
bc -> e
ac -> de
ab -> cde
a -> e
meet-irreducibles: b cd ae bce ade bde acde bcde 
max antichain: cd ae bce bde 
dimension: 4

339.
id: 4294901761
d -> e
c -> e
b -> e
a -> e
meet-irreducibles: abce abde acde bcde 
max antichain: abce abde acde bcde 
dimension: 4

340.
id: 4279571733
bc -> de
a -> de
meet-irreducibles: bd cd be ce abde acde bcde 
max antichain: bd cd be ce 
dimension: 4

341.
id: 4149548309
bc -> e
ab -> cde
a -> de
meet-irreducibles: bd cd bce ade bde acde bcde 
max antichain: bd cd bce ade 
dimension: 4

342.
id: 4145485077
bc -> de
ac -> de
ab -> cde
a -> e
meet-irreducibles: bd cd ae be ce ade bde acde bcde 
max antichain: bd cd ae be ce 
dimension: 5

343.
id: 4114027797
ad -> ce
bc -> de
ab -> cde
a -> e
meet-irreducibles: bd cd ae be ace bde acde bcde 
max antichain: bd cd ae be 
dimension: 4

344.
id: 4279571221
bd -> e
bc -> de
ac -> de
ab -> de
a -> d
meet-irreducibles: b ad cd be ce abde acde bcde 
max antichain: b ad cd ce 
dimension: 4

345.
id: 4149547797
bd -> e
bc -> e
ac -> de
ab -> cde
a -> d
meet-irreducibles: b ad cd bce ade bde acde bcde 
max antichain: ad cd bce bde 
dimension: 4

346.
id: 4286513169
d -> e
abc -> de
b -> e
a -> e
meet-irreducibles: c abe ace bce abde acde bcde 
max antichain: abe ace bce 
dimension: 3

347.
id: 4283892483
ac -> de
c -> e
ab -> de
b -> e
meet-irreducibles: ad ae bce abde acde bcde 
max antichain: ad ae bce 
dimension: 3

348.
id: 4116002069
bc -> e
ab -> cde
a -> cd
meet-irreducibles: bd acd bce bde acde bcde 
max antichain: bd acd bce 
dimension: 3

349.
id: 4283891987
cd -> e
ad -> e
ac -> de
ab -> de
b -> e
meet-irreducibles: a c d ae bce abde acde bcde 
max antichain: a c d 
dimension: 3

350.
id: 4149674071
d -> e
ac -> de
ab -> cde
meet-irreducibles: a bc ae bce ade bde acde bcde 
max antichain: a bc bde 
dimension: 3

351.
id: 4286513409
abc -> de
c -> e
b -> e
a -> e
meet-irreducibles: d abe ace bce abde acde bcde 
max antichain: d abe ace bce 
dimension: 4

352.
id: 4113892663
ae -> c
cd -> e
ad -> ce
bc -> de
ab -> cde
meet-irreducibles: a ac bd be ace bde acde bcde 
max antichain: a bd be 
dimension: 3

353.
id: 4116001557
ae -> cd
bd -> e
bc -> e
ab -> cde
a -> d
meet-irreducibles: b ad acd bce bde acde bcde 
max antichain: ad bce bde 
dimension: 3

354.
id: 4279567123
ae -> d
cd -> e
bc -> de
ac -> de
ab -> de
b -> e
meet-irreducibles: a c ad be ce abde acde bcde 
max antichain: a c be 
dimension: 3

355.
id: 4149543699
ae -> d
cd -> e
ac -> de
ab -> cde
b -> e
meet-irreducibles: a c ad bce ade bde acde bcde 
max antichain: a c bde 
dimension: 3

356.
id: 4283891985
cd -> e
ac -> de
ab -> de
b -> e
a -> e
meet-irreducibles: c d ae bce abde acde bcde 
max antichain: c d ae 
dimension: 3

357.
id: 4151771409
cd -> e
ab -> cde
b -> e
a -> e
meet-irreducibles: c d ace bce ade bde acde bcde 
max antichain: ace bce ade bde 
dimension: 4

358.
id: 4151771143
d -> e
c -> e
ab -> cde
meet-irreducibles: a b ace bce ade bde acde bcde 
max antichain: ace bce ade bde 
dimension: 4

359.
id: 4283760725
d -> e
a -> de
meet-irreducibles: bc bce abde acde bcde 
max antichain: bc abde acde 
dimension: 3

360.
id: 4149674069
d -> e
ac -> de
ab -> cde
a -> e
meet-irreducibles: bc ae bce ade bde acde bcde 
max antichain: bc ae bde 
dimension: 3

361.
id: 4113892661
cd -> e
ad -> ce
bc -> de
ab -> cde
a -> c
meet-irreducibles: ac bd be ace bde acde bcde 
max antichain: ac bd be 
dimension: 3

362.
id: 4283761425
cd -> e
ac -> de
ab -> de
b -> e
a -> d
meet-irreducibles: c ad bce abde acde bcde 
max antichain: bce abde acde 
dimension: 3

363.
id: 4279697431
d -> e
bc -> de
ac -> de
ab -> de
meet-irreducibles: a b c ae be ce abde acde bcde 
max antichain: a b c 
dimension: 3

364.
id: 4149674007
d -> e
bc -> e
ac -> de
ab -> cde
meet-irreducibles: a b c ae bce ade bde acde bcde 
max antichain: a b c 
dimension: 3

365.
id: 4285988869
d -> e
c -> e
ab -> de
a -> e
meet-irreducibles: b ace bce abde acde bcde 
max antichain: ace bce abde 
dimension: 3

366.
id: 4149674247
bd -> e
ad -> e
ac -> de
c -> e
ab -> cde
meet-irreducibles: a b d ae bce ade bde acde bcde 
max antichain: a b d 
dimension: 3

367.
id: 4149543253
cd -> e
bd -> e
ab -> cde
a -> de
meet-irreducibles: bc d bce ade bde acde bcde 
max antichain: bc ade bde 
dimension: 3

368.
id: 4283891733
d -> e
bc -> e
ac -> de
ab -> de
a -> e
meet-irreducibles: b c ae bce abde acde bcde 
max antichain: b c ae 
dimension: 3

369.
id: 4113895733
bd -> e
ad -> ce
bc -> de
ab -> cde
a -> c
meet-irreducibles: b ac cd be ace bde acde bcde 
max antichain: b ac cd 
dimension: 3

370.
id: 4145480967
ad -> e
bc -> de
ac -> de
c -> e
ab -> cde
meet-irreducibles: a bd ae be ce ade bde acde bcde 
max antichain: a bd be ce 
dimension: 4

371.
id: 4114023687
ad -> ce
bc -> de
c -> e
ab -> cde
meet-irreducibles: a bd ae be ace bde acde bcde 
max antichain: a bd be 
dimension: 3

372.
id: 4115989845
cd -> e
a -> cde
meet-irreducibles: bc bd bce bde acde bcde 
max antichain: bc bd acde 
dimension: 3

373.
id: 4283765009
b -> e
a -> de
meet-irreducibles: cd bce abde acde bcde 
max antichain: cd bce abde 
dimension: 3

374.
id: 4279701777
bc -> de
ac -> de
ab -> de
b -> e
a -> e
meet-irreducibles: cd ae be ce abde acde bcde 
max antichain: cd ae be ce 
dimension: 4

375.
id: 4149678353
ac -> de
ab -> cde
b -> e
a -> e
meet-irreducibles: cd ae bce ade bde acde bcde 
max antichain: cd ae bce bde 
dimension: 4

376.
id: 4149543687
ae -> d
bd -> e
ac -> de
c -> e
ab -> cde
meet-irreducibles: a b ad bce ade bde acde bcde 
max antichain: a bce bde 
dimension: 3

377.
id: 4113634613
be -> d
ad -> ce
bc -> de
ab -> cde
a -> c
meet-irreducibles: b ac bd cd ace bde acde bcde 
max antichain: b ac cd 
dimension: 3

378.
id: 4145479959
cd -> e
bd -> e
ad -> e
bc -> de
ac -> de
ab -> cde
meet-irreducibles: a b c d ae be ce ade bde acde bcde 
max antichain: a b c d 
dimension: 4

379.
id: 4114022679
cd -> e
bd -> e
ad -> ce
bc -> de
ac -> e
ab -> cde
meet-irreducibles: a b c d ae be ace bde acde bcde 
max antichain: a b c d 
dimension: 4

380.
id: 4151771397
bd -> e
c -> e
ab -> cde
a -> e
meet-irreducibles: b d ace bce ade bde acde bcde 
max antichain: ace bce ade bde 
dimension: 4

381.
id: 4145350407
ae -> d
bc -> de
ac -> de
c -> e
ab -> cde
meet-irreducibles: a ad bd be ce ade bde acde bcde 
max antichain: a bd be ce 
dimension: 4

382.
id: 4144432903
c -> de
ab -> cde
meet-irreducibles: ad bd ae be ade bde acde bcde 
max antichain: ad bd ae be 
dimension: 4

383.
id: 4279571217
bc -> de
ac -> de
ab -> de
b -> e
a -> d
meet-irreducibles: ad cd be ce abde acde bcde 
max antichain: ad cd be ce 
dimension: 4

384.
id: 4149547793
ac -> de
ab -> cde
b -> e
a -> d
meet-irreducibles: ad cd bce ade bde acde bcde 
max antichain: ad cd bce bde 
dimension: 4

385.
id: 4149675269
ac -> de
c -> e
ab -> cde
a -> e
meet-irreducibles: bd ae bce ade bde acde bcde 
max antichain: bd ae bce 
dimension: 3

386.
id: 4145349399
ae -> d
cd -> e
bd -> e
bc -> de
ac -> de
ab -> cde
meet-irreducibles: a b c ad be ce ade bde acde bcde 
max antichain: a b c 
dimension: 3

387.
id: 4283760917
cd -> e
bd -> e
bc -> e
a -> de
meet-irreducibles: b c d bce abde acde bcde 
max antichain: b c d 
dimension: 3

388.
id: 4279697685
cd -> e
bd -> e
bc -> de
ac -> de
ab -> de
a -> e
meet-irreducibles: b c d ae be ce abde acde bcde 
max antichain: b c d ae 
dimension: 4

389.
id: 4149674261
cd -> e
bd -> e
bc -> e
ac -> de
ab -> cde
a -> e
meet-irreducibles: b c d ae bce ade bde acde bcde 
max antichain: b c d ae 
dimension: 4

390.
id: 4145088279
be -> d
ae -> d
cd -> e
bc -> de
ac -> de
ab -> cde
meet-irreducibles: a b c ad bd ce ade bde acde bcde 
max antichain: a b c 
dimension: 3

391.
id: 4149544709
ac -> de
c -> e
ab -> cde
a -> d
meet-irreducibles: ad bd bce ade bde acde bcde 
max antichain: ad bd bce 
dimension: 3

392.
id: 4149544213
cd -> e
bc -> e
ab -> cde
a -> de
meet-irreducibles: c bd bce ade bde acde bcde 
max antichain: c bd ade 
dimension: 3

393.
id: 4116001553
ae -> cd
ab -> cde
b -> e
a -> d
meet-irreducibles: ad acd bce bde acde bcde 
max antichain: ad bce bde 
dimension: 3

394.
id: 4145480981
cd -> e
bc -> de
ac -> de
ab -> cde
a -> e
meet-irreducibles: c bd ae be ce ade bde acde bcde 
max antichain: c bd ae be 
dimension: 4

395.
id: 4114023701
cd -> e
ad -> ce
bc -> de
ab -> cde
a -> e
meet-irreducibles: c bd ae be ace bde acde bcde 
max antichain: c bd ae be 
dimension: 4

396.
id: 3579123029
bcd -> e
a -> bcde
meet-irreducibles: bc bd cd bce bde cde bcde 
max antichain: bc bd cd 
dimension: 3

397.
id: 4279567125
cd -> e
bd -> e
bc -> de
ac -> de
ab -> de
a -> d
meet-irreducibles: b c ad be ce abde acde bcde 
max antichain: b c ad 
dimension: 3

398.
id: 4149543701
cd -> e
bd -> e
bc -> e
ac -> de
ab -> cde
a -> d
meet-irreducibles: b c ad bce ade bde acde bcde 
max antichain: b c ad 
dimension: 3

399.
id: 4145350421
cd -> e
bc -> de
ac -> de
ab -> cde
a -> d
meet-irreducibles: c ad bd be ce ade bde acde bcde 
max antichain: c ad bd be 
dimension: 4

400.
id: 4283891731
d -> e
ac -> de
ab -> de
b -> e
meet-irreducibles: a c ae bce abde acde bcde 
max antichain: bce abde acde 
dimension: 3

401.
id: 4113895731
ae -> c
ad -> ce
bc -> de
ab -> cde
b -> e
meet-irreducibles: a ac cd be ace bde acde bcde 
max antichain: a cd be 
dimension: 3

402.
id: 4279570709
bd -> e
bc -> de
a -> de
meet-irreducibles: b cd be ce abde acde bcde 
max antichain: b cd ce 
dimension: 3

403.
id: 4149547285
bd -> e
bc -> e
ab -> cde
a -> de
meet-irreducibles: b cd bce ade bde acde bcde 
max antichain: cd bce ade bde 
dimension: 4

404.
id: 4145484053
bd -> e
bc -> de
ac -> de
ab -> cde
a -> e
meet-irreducibles: b cd ae be ce ade bde acde bcde 
max antichain: b cd ae ce 
dimension: 4

405.
id: 4114026773
bd -> e
ad -> ce
bc -> de
ab -> cde
a -> e
meet-irreducibles: b cd ae be ace bde acde bcde 
max antichain: b cd ae 
dimension: 3

406.
id: 4286513153
d -> e
abc -> de
c -> e
b -> e
a -> e
meet-irreducibles: abe ace bce abde acde bcde 
max antichain: abe ace bce 
dimension: 3

407.
id: 4145354005
bc -> de
ab -> cde
a -> de
meet-irreducibles: bd cd be ce ade bde acde bcde 
max antichain: bd cd be ce ade 
dimension: 5

408.
id: 4115993877
bc -> e
a -> cde
meet-irreducibles: bd cd bce bde acde bcde 
max antichain: bd cd bce 
dimension: 3

409.
id: 4113896725
bc -> de
ab -> cde
a -> ce
meet-irreducibles: bd cd be ace bde acde bcde 
max antichain: bd cd be ace 
dimension: 4

410.
id: 4283891971
ad -> e
ac -> de
c -> e
ab -> de
b -> e
meet-irreducibles: a d ae bce abde acde bcde 
max antichain: a d bce 
dimension: 3

411.
id: 4145353493
bd -> e
bc -> de
ac -> de
ab -> cde
a -> d
meet-irreducibles: b ad cd be ce ade bde acde bcde 
max antichain: b ad cd ce 
dimension: 4

412.
id: 4285988881
d -> e
ab -> de
b -> e
a -> e
meet-irreducibles: c ace bce abde acde bcde 
max antichain: ace bce abde 
dimension: 3

413.
id: 4116001045
bd -> e
bc -> e
ab -> cde
a -> cd
meet-irreducibles: b acd bce bde acde bcde 
max antichain: acd bce bde 
dimension: 3

414.
id: 4283761411
ae -> d
ac -> de
c -> e
ab -> de
b -> e
meet-irreducibles: a ad bce abde acde bcde 
max antichain: bce abde acde 
dimension: 3

415.
id: 4113891639
ae -> c
cd -> e
bd -> e
ad -> ce
bc -> de
ab -> cde
meet-irreducibles: a b ac d be ace bde acde bcde 
max antichain: a b d 
dimension: 3

416.
id: 4149674755
ac -> de
c -> e
ab -> cde
b -> e
meet-irreducibles: ad ae bce ade bde acde bcde 
max antichain: ad ae bce bde 
dimension: 4

417.
id: 4149674259
cd -> e
ad -> e
ac -> de
ab -> cde
b -> e
meet-irreducibles: a c d ae bce ade bde acde bcde 
max antichain: a c d 
dimension: 3

418.
id: 4285989121
c -> e
ab -> de
b -> e
a -> e
meet-irreducibles: d ace bce abde acde bcde 
max antichain: d ace bce 
dimension: 3

419.
id: 4113630519
be -> d
ae -> c
cd -> e
ad -> ce
bc -> de
ab -> cde
meet-irreducibles: a b ac bd ace bde acde bcde 
max antichain: a b 
dimension: 2

420.
id: 4145349395
ae -> d
cd -> e
bc -> de
ac -> de
ab -> cde
b -> e
meet-irreducibles: a c ad be ce ade bde acde bcde 
max antichain: a c be 
dimension: 3

421.
id: 4283761409
ac -> de
c -> e
ab -> de
b -> e
a -> d
meet-irreducibles: ad bce abde acde bcde 
max antichain: bce abde acde 
dimension: 3

422.
id: 4113891637
cd -> e
bd -> e
ad -> ce
bc -> de
ab -> cde
a -> c
meet-irreducibles: b ac d be ace bde acde bcde 
max antichain: b ac d 
dimension: 3

423.
id: 4283760913
cd -> e
b -> e
a -> de
meet-irreducibles: c d bce abde acde bcde 
max antichain: bce abde acde 
dimension: 3

424.
id: 4279697681
cd -> e
bc -> de
ac -> de
ab -> de
b -> e
a -> e
meet-irreducibles: c d ae be ce abde acde bcde 
max antichain: c d ae be 
dimension: 4

425.
id: 4149674257
cd -> e
ac -> de
ab -> cde
b -> e
a -> e
meet-irreducibles: c d ae bce ade bde acde bcde 
max antichain: c d ae 
dimension: 3

426.
id: 4149673991
d -> e
ac -> de
c -> e
ab -> cde
meet-irreducibles: a b ae bce ade bde acde bcde 
max antichain: a bce bde 
dimension: 3

427.
id: 4149542997
d -> e
ab -> cde
a -> de
meet-irreducibles: bc bce ade bde acde bcde 
max antichain: bc ade bde 
dimension: 3

428.
id: 4113630517
be -> d
cd -> e
ad -> ce
bc -> de
ab -> cde
a -> c
meet-irreducibles: b ac bd ace bde acde bcde 
max antichain: b ac 
dimension: 2

429.
id: 4279567121
cd -> e
bc -> de
ac -> de
ab -> de
b -> e
a -> d
meet-irreducibles: c ad be ce abde acde bcde 
max antichain: c ad be 
dimension: 3

430.
id: 4149543697
cd -> e
ac -> de
ab -> cde
b -> e
a -> d
meet-irreducibles: c ad bce ade bde acde bcde 
max antichain: c ad bde 
dimension: 3

431.
id: 4145479703
d -> e
bc -> de
ac -> de
ab -> cde
meet-irreducibles: a b c ae be ce ade bde acde bcde 
max antichain: a b c 
dimension: 3

432.
id: 4151771141
d -> e
c -> e
ab -> cde
a -> e
meet-irreducibles: b ace bce ade bde acde bcde 
max antichain: ace bce ade bde 
dimension: 4

433.
id: 4145479943
bd -> e
ad -> e
bc -> de
ac -> de
c -> e
ab -> cde
meet-irreducibles: a b d ae be ce ade bde acde bcde 
max antichain: a b d ce 
dimension: 4

434.
id: 4114022663
bd -> e
ad -> ce
bc -> de
c -> e
ab -> cde
meet-irreducibles: a b d ae be ace bde acde bcde 
max antichain: a b d 
dimension: 3

435.
id: 4283760661
d -> e
bc -> e
a -> de
meet-irreducibles: b c bce abde acde bcde 
max antichain: bce abde acde 
dimension: 3

436.
id: 4279697429
d -> e
bc -> de
ac -> de
ab -> de
a -> e
meet-irreducibles: b c ae be ce abde acde bcde 
max antichain: b c ae 
dimension: 3

437.
id: 4149674005
d -> e
bc -> e
ac -> de
ab -> cde
a -> e
meet-irreducibles: b c ae bce ade bde acde bcde 
max antichain: b c ae 
dimension: 3

438.
id: 4144432391
ad -> e
c -> de
ab -> cde
meet-irreducibles: a bd ae be ade bde acde bcde 
max antichain: a bd be 
dimension: 3

439.
id: 4279570705
bc -> de
b -> e
a -> de
meet-irreducibles: cd be ce abde acde bcde 
max antichain: cd be ce 
dimension: 3

440.
id: 4149547281
ab -> cde
b -> e
a -> de
meet-irreducibles: cd bce ade bde acde bcde 
max antichain: cd bce ade bde 
dimension: 4

441.
id: 4145484049
bc -> de
ac -> de
ab -> cde
b -> e
a -> e
meet-irreducibles: cd ae be ce ade bde acde bcde 
max antichain: cd ae be ce 
dimension: 4

442.
id: 4114026769
ad -> ce
bc -> de
ab -> cde
b -> e
a -> e
meet-irreducibles: cd ae be ace bde acde bcde 
max antichain: cd ae be 
dimension: 3

443.
id: 4145349383
ae -> d
bd -> e
bc -> de
ac -> de
c -> e
ab -> cde
meet-irreducibles: a b ad be ce ade bde acde bcde 
max antichain: a b ce 
dimension: 3

444.
id: 4149674245
bd -> e
ac -> de
c -> e
ab -> cde
a -> e
meet-irreducibles: b d ae bce ade bde acde bcde 
max antichain: b d ae 
dimension: 3

445.
id: 4145088263
be -> d
ae -> d
bc -> de
ac -> de
c -> e
ab -> cde
meet-irreducibles: a b ad bd ce ade bde acde bcde 
max antichain: a b ce 
dimension: 3

446.
id: 4145353489
bc -> de
ac -> de
ab -> cde
b -> e
a -> d
meet-irreducibles: ad cd be ce ade bde acde bcde 
max antichain: ad cd be ce 
dimension: 4

447.
id: 4149544197
c -> e
ab -> cde
a -> de
meet-irreducibles: bd bce ade bde acde bcde 
max antichain: bd bce ade 
dimension: 3

448.
id: 4145480965
bc -> de
ac -> de
c -> e
ab -> cde
a -> e
meet-irreducibles: bd ae be ce ade bde acde bcde 
max antichain: bd ae be ce 
dimension: 4

449.
id: 4114023685
ad -> ce
bc -> de
c -> e
ab -> cde
a -> e
meet-irreducibles: bd ae be ace bde acde bcde 
max antichain: bd ae be 
dimension: 3

450.
id: 4116001041
ab -> cde
b -> e
a -> cd
meet-irreducibles: acd bce bde acde bcde 
max antichain: acd bce bde 
dimension: 3

451.
id: 4149543685
bd -> e
ac -> de
c -> e
ab -> cde
a -> d
meet-irreducibles: b ad bce ade bde acde bcde 
max antichain: ad bce bde 
dimension: 3

452.
id: 4113891635
ae -> c
cd -> e
ad -> ce
bc -> de
ab -> cde
b -> e
meet-irreducibles: a ac d be ace bde acde bcde 
max antichain: a d be 
dimension: 3

453.
id: 4279566613
cd -> e
bd -> e
bc -> de
a -> de
meet-irreducibles: b c d be ce abde acde bcde 
max antichain: b c d 
dimension: 3

454.
id: 4149543189
cd -> e
bd -> e
bc -> e
ab -> cde
a -> de
meet-irreducibles: b c d bce ade bde acde bcde 
max antichain: b c d 
dimension: 3

455.
id: 4145479957
cd -> e
bd -> e
bc -> de
ac -> de
ab -> cde
a -> e
meet-irreducibles: b c d ae be ce ade bde acde bcde 
max antichain: b c d ae 
dimension: 4

456.
id: 4114022677
cd -> e
bd -> e
ad -> ce
bc -> de
ab -> cde
a -> e
meet-irreducibles: b c d ae be ace bde acde bcde 
max antichain: b c d ae 
dimension: 4

457.
id: 4145350405
bc -> de
ac -> de
c -> e
ab -> cde
a -> d
meet-irreducibles: ad bd be ce ade bde acde bcde 
max antichain: ad bd be ce 
dimension: 4

458.
id: 4279305493
be -> d
cd -> e
bc -> de
a -> de
meet-irreducibles: b c bd ce abde acde bcde 
max antichain: abde acde bcde 
dimension: 3

459.
id: 4145349909
cd -> e
bc -> de
ab -> cde
a -> de
meet-irreducibles: c bd be ce ade bde acde bcde 
max antichain: c bd be ade 
dimension: 4

460.
id: 4115989781
cd -> e
bc -> e
a -> cde
meet-irreducibles: c bd bce bde acde bcde 
max antichain: bd bce acde 
dimension: 3

461.
id: 4113892629
cd -> e
bc -> de
ab -> cde
a -> ce
meet-irreducibles: c bd be ace bde acde bcde 
max antichain: c bd be 
dimension: 3

462.
id: 4145349397
cd -> e
bd -> e
bc -> de
ac -> de
ab -> cde
a -> d
meet-irreducibles: b c ad be ce ade bde acde bcde 
max antichain: b c ad 
dimension: 3

463.
id: 4283891715
d -> e
ac -> de
c -> e
ab -> de
b -> e
meet-irreducibles: a ae bce abde acde bcde 
max antichain: bce abde acde 
dimension: 3

464.
id: 4145088277
be -> d
cd -> e
bc -> de
ac -> de
ab -> cde
a -> d
meet-irreducibles: b c ad bd ce ade bde acde bcde 
max antichain: b c ad 
dimension: 3

465.
id: 4113891383
ae -> c
ad -> ce
d -> e
bc -> de
ab -> cde
meet-irreducibles: a b ac be ace bde acde bcde 
max antichain: a b 
dimension: 2

466.
id: 4149674003
d -> e
ac -> de
ab -> cde
b -> e
meet-irreducibles: a c ae bce ade bde acde bcde 
max antichain: a c bde 
dimension: 3

467.
id: 4145352981
bd -> e
bc -> de
ab -> cde
a -> de
meet-irreducibles: b cd be ce ade bde acde bcde 
max antichain: b cd ce ade 
dimension: 4

468.
id: 4115992853
bd -> e
bc -> e
a -> cde
meet-irreducibles: b cd bce bde acde bcde 
max antichain: cd bce bde 
dimension: 3

469.
id: 4113895701
bd -> e
bc -> de
ab -> cde
a -> ce
meet-irreducibles: b cd be ace bde acde bcde 
max antichain: b cd ace 
dimension: 3

470.
id: 4285988865
d -> e
c -> e
ab -> de
b -> e
a -> e
meet-irreducibles: ace bce abde acde bcde 
max antichain: ace bce abde 
dimension: 3

471.
id: 4113634581
be -> d
bc -> de
ab -> cde
a -> ce
meet-irreducibles: b bd cd ace bde acde bcde 
max antichain: b cd ace 
dimension: 3

472.
id: 4111799573
bc -> de
a -> cde
meet-irreducibles: bd cd be ce bde acde bcde 
max antichain: bd cd be ce 
dimension: 4

473.
id: 3579122965
bc -> e
a -> bcde
meet-irreducibles: bd cd bce bde cde bcde 
max antichain: bd cd bce 
dimension: 3

474.
id: 4149674243
ad -> e
ac -> de
c -> e
ab -> cde
b -> e
meet-irreducibles: a d ae bce ade bde acde bcde 
max antichain: a d bce 
dimension: 3

475.
id: 4283891729
d -> e
ac -> de
ab -> de
b -> e
a -> e
meet-irreducibles: c ae bce abde acde bcde 
max antichain: bce abde acde 
dimension: 3

476.
id: 4113895729
ad -> ce
bc -> de
ab -> cde
b -> e
a -> c
meet-irreducibles: ac cd be ace bde acde bcde 
max antichain: ac cd be 
dimension: 3

477.
id: 4279567107
ae -> d
bc -> de
ac -> de
c -> e
ab -> de
b -> e
meet-irreducibles: a ad be ce abde acde bcde 
max antichain: a be ce 
dimension: 3

478.
id: 4149543683
ae -> d
ac -> de
c -> e
ab -> cde
b -> e
meet-irreducibles: a ad bce ade bde acde bcde 
max antichain: a bce bde 
dimension: 3

479.
id: 4113634609
ad -> ce
bc -> de
ab -> cde
b -> d
a -> c
meet-irreducibles: ac bd cd ace bde acde bcde 
max antichain: ac bd cd 
dimension: 3

480.
id: 4283891969
ac -> de
c -> e
ab -> de
b -> e
a -> e
meet-irreducibles: d ae bce abde acde bcde 
max antichain: d ae bce 
dimension: 3

481.
id: 4151771393
c -> e
ab -> cde
b -> e
a -> e
meet-irreducibles: d ace bce ade bde acde bcde 
max antichain: ace bce ade bde 
dimension: 4

482.
id: 4279567105
bc -> de
ac -> de
c -> e
ab -> de
b -> e
a -> d
meet-irreducibles: ad be ce abde acde bcde 
max antichain: ad be ce 
dimension: 3

483.
id: 4149543681
ac -> de
c -> e
ab -> cde
b -> e
a -> d
meet-irreducibles: ad bce ade bde acde bcde 
max antichain: ad bce bde 
dimension: 3

484.
id: 4279566609
cd -> e
bc -> de
b -> e
a -> de
meet-irreducibles: c d be ce abde acde bcde 
max antichain: c d be 
dimension: 3

485.
id: 4149543185
cd -> e
ab -> cde
b -> e
a -> de
meet-irreducibles: c d bce ade bde acde bcde 
max antichain: c ade bde 
dimension: 3

486.
id: 4145479953
cd -> e
bc -> de
ac -> de
ab -> cde
b -> e
a -> e
meet-irreducibles: c d ae be ce ade bde acde bcde 
max antichain: c d ae be 
dimension: 4

487.
id: 4114022673
cd -> e
ad -> ce
bc -> de
ab -> cde
b -> e
a -> e
meet-irreducibles: c d ae be ace bde acde bcde 
max antichain: c d ae be 
dimension: 4

488.
id: 4145479687
d -> e
bc -> de
ac -> de
c -> e
ab -> cde
meet-irreducibles: a b ae be ce ade bde acde bcde 
max antichain: a b ce 
dimension: 3

489.
id: 4114022407
ad -> ce
d -> e
bc -> de
c -> e
ab -> cde
meet-irreducibles: a b ae be ace bde acde bcde 
max antichain: a b 
dimension: 2

490.
id: 4279305489
cd -> e
bc -> de
b -> d
a -> de
meet-irreducibles: c bd ce abde acde bcde 
max antichain: abde acde bcde 
dimension: 3

491.
id: 4145349393
cd -> e
bc -> de
ac -> de
ab -> cde
b -> e
a -> d
meet-irreducibles: c ad be ce ade bde acde bcde 
max antichain: c ad be 
dimension: 3

492.
id: 4149673989
d -> e
ac -> de
c -> e
ab -> cde
a -> e
meet-irreducibles: b ae bce ade bde acde bcde 
max antichain: ae bce bde 
dimension: 3

493.
id: 4145088273
cd -> e
bc -> de
ac -> de
ab -> cde
b -> d
a -> d
meet-irreducibles: c ad bd ce ade bde acde bcde 
max antichain: c ad bd 
dimension: 3

494.
id: 4144431367
bd -> e
ad -> e
c -> de
ab -> cde
meet-irreducibles: a b d ae be ade bde acde bcde 
max antichain: a b d 
dimension: 3

495.
id: 4113891379
ae -> c
ad -> ce
d -> e
bc -> de
ab -> cde
b -> e
meet-irreducibles: a ac be ace bde acde bcde 
max antichain: a be 
dimension: 2

496.
id: 4279566357
d -> e
bc -> de
a -> de
meet-irreducibles: b c be ce abde acde bcde 
max antichain: abde acde bcde 
dimension: 3

497.
id: 4149542933
d -> e
bc -> e
ab -> cde
a -> de
meet-irreducibles: b c bce ade bde acde bcde 
max antichain: b c ade 
dimension: 3

498.
id: 4145479701
d -> e
bc -> de
ac -> de
ab -> cde
a -> e
meet-irreducibles: b c ae be ce ade bde acde bcde 
max antichain: b c ae 
dimension: 3

499.
id: 4279308561
b -> de
a -> de
meet-irreducibles: cd ce abde acde bcde 
max antichain: cd ce abde 
dimension: 3

500.
id: 4145352977
bc -> de
ab -> cde
b -> e
a -> de
meet-irreducibles: cd be ce ade bde acde bcde 
max antichain: cd be ce ade 
dimension: 4

501.
id: 4115992849
b -> e
a -> cde
meet-irreducibles: cd bce bde acde bcde 
max antichain: cd bce bde 
dimension: 3

502.
id: 4113895697
bc -> de
ab -> cde
b -> e
a -> ce
meet-irreducibles: cd be ace bde acde bcde 
max antichain: cd be ace 
dimension: 3

503.
id: 4144300807
ae -> d
bd -> e
c -> de
ab -> cde
meet-irreducibles: a b ad be ade bde acde bcde 
max antichain: a b 
dimension: 2

504.
id: 4113634577
bc -> de
ab -> cde
b -> d
a -> ce
meet-irreducibles: bd cd ace bde acde bcde 
max antichain: bd cd ace 
dimension: 3

505.
id: 4149543173
bd -> e
c -> e
ab -> cde
a -> de
meet-irreducibles: b d bce ade bde acde bcde 
max antichain: bce ade bde 
dimension: 3

506.
id: 4145479941
bd -> e
bc -> de
ac -> de
c -> e
ab -> cde
a -> e
meet-irreducibles: b d ae be ce ade bde acde bcde 
max antichain: b d ae ce 
dimension: 4

507.
id: 4114022661
bd -> e
ad -> ce
bc -> de
c -> e
ab -> cde
a -> e
meet-irreducibles: b d ae be ace bde acde bcde 
max antichain: b d ae 
dimension: 3

508.
id: 4145349893
bc -> de
c -> e
ab -> cde
a -> de
meet-irreducibles: bd be ce ade bde acde bcde 
max antichain: bd be ce ade 
dimension: 4

509.
id: 4115989765
c -> e
a -> cde
meet-irreducibles: bd bce bde acde bcde 
max antichain: bd bce acde 
dimension: 3

510.
id: 4113892613
bc -> de
c -> e
ab -> cde
a -> ce
meet-irreducibles: bd be ace bde acde bcde 
max antichain: bd be ace 
dimension: 3

511.
id: 4144432389
c -> de
ab -> cde
a -> e
meet-irreducibles: bd ae be ade bde acde bcde 
max antichain: bd ae be 
dimension: 3

512.
id: 4145349381
bd -> e
bc -> de
ac -> de
c -> e
ab -> cde
a -> d
meet-irreducibles: b ad be ce ade bde acde bcde 
max antichain: b ad ce 
dimension: 3

513.
id: 4145348885
cd -> e
bd -> e
bc -> de
ab -> cde
a -> de
meet-irreducibles: b c d be ce ade bde acde bcde 
max antichain: b c d 
dimension: 3

514.
id: 4115988757
cd -> e
bd -> e
bc -> e
a -> cde
meet-irreducibles: b c d bce bde acde bcde 
max antichain: b c d 
dimension: 3

515.
id: 4113891605
cd -> e
bd -> e
bc -> de
ab -> cde
a -> ce
meet-irreducibles: b c d be ace bde acde bcde 
max antichain: b c d 
dimension: 3

516.
id: 4145088261
be -> d
bc -> de
ac -> de
c -> e
ab -> cde
a -> d
meet-irreducibles: b ad bd ce ade bde acde bcde 
max antichain: b ad ce 
dimension: 3

517.
id: 4145087765
be -> d
cd -> e
bc -> de
ab -> cde
a -> de
meet-irreducibles: b c bd ce ade bde acde bcde 
max antichain: b c ade 
dimension: 3

518.
id: 4113630485
be -> d
cd -> e
bc -> de
ab -> cde
a -> ce
meet-irreducibles: b c bd ace bde acde bcde 
max antichain: b c 
dimension: 2

519.
id: 4111795477
cd -> e
bc -> de
a -> cde
meet-irreducibles: c bd be ce bde acde bcde 
max antichain: c bd be 
dimension: 3

520.
id: 4113891633
cd -> e
ad -> ce
bc -> de
ab -> cde
b -> e
a -> c
meet-irreducibles: ac d be ace bde acde bcde 
max antichain: ac d be 
dimension: 3

521.
id: 4149673987
d -> e
ac -> de
c -> e
ab -> cde
b -> e
meet-irreducibles: a ae bce ade bde acde bcde 
max antichain: a bce bde 
dimension: 3

522.
id: 4113630513
cd -> e
ad -> ce
bc -> de
ab -> cde
b -> d
a -> c
meet-irreducibles: ac bd ace bde acde bcde 
max antichain: ac bd 
dimension: 2

523.
id: 4111798549
bd -> e
bc -> de
a -> cde
meet-irreducibles: b cd be ce bde acde bcde 
max antichain: b cd ce 
dimension: 3

524.
id: 3579121941
bd -> e
bc -> e
a -> bcde
meet-irreducibles: b cd bce bde cde bcde 
max antichain: cd bce bde 
dimension: 3

525.
id: 4283891713
d -> e
ac -> de
c -> e
ab -> de
b -> e
a -> e
meet-irreducibles: ae bce abde acde bcde 
max antichain: bce abde acde 
dimension: 3

526.
id: 4151771137
d -> e
c -> e
ab -> cde
b -> e
a -> e
meet-irreducibles: ace bce ade bde acde bcde 
max antichain: ace bce ade bde 
dimension: 4

527.
id: 3574928661
bc -> de
a -> bcde
meet-irreducibles: bd cd be ce bde cde bcde 
max antichain: bd cd be ce 
dimension: 4

528.
id: 4114022659
ad -> ce
bc -> de
c -> e
ab -> cde
b -> e
meet-irreducibles: a d ae be ace bde acde bcde 
max antichain: a d be 
dimension: 3

529.
id: 4113891381
ad -> ce
d -> e
bc -> de
ab -> cde
a -> c
meet-irreducibles: b ac be ace bde acde bcde 
max antichain: b ac 
dimension: 2

530.
id: 4283760657
d -> e
b -> e
a -> de
meet-irreducibles: c bce abde acde bcde 
max antichain: bce abde acde 
dimension: 3

531.
id: 4279697425
d -> e
bc -> de
ac -> de
ab -> de
b -> e
a -> e
meet-irreducibles: c ae be ce abde acde bcde 
max antichain: c ae be 
dimension: 3

532.
id: 4149674001
d -> e
ac -> de
ab -> cde
b -> e
a -> e
meet-irreducibles: c ae bce ade bde acde bcde 
max antichain: c ae bde 
dimension: 3

533.
id: 4145349379
ae -> d
bc -> de
ac -> de
c -> e
ab -> cde
b -> e
meet-irreducibles: a ad be ce ade bde acde bcde 
max antichain: a be ce 
dimension: 3

534.
id: 4283760897
c -> e
b -> e
a -> de
meet-irreducibles: d bce abde acde bcde 
max antichain: bce abde acde 
dimension: 3

535.
id: 4279697665
bc -> de
ac -> de
c -> e
ab -> de
b -> e
a -> e
meet-irreducibles: d ae be ce abde acde bcde 
max antichain: d ae be ce 
dimension: 4

536.
id: 4149674241
ac -> de
c -> e
ab -> cde
b -> e
a -> e
meet-irreducibles: d ae bce ade bde acde bcde 
max antichain: d ae bce 
dimension: 3

537.
id: 4279305473
bc -> de
c -> e
b -> d
a -> de
meet-irreducibles: bd ce abde acde bcde 
max antichain: abde acde bcde 
dimension: 3

538.
id: 4145349377
bc -> de
ac -> de
c -> e
ab -> cde
b -> e
a -> d
meet-irreducibles: ad be ce ade bde acde bcde 
max antichain: ad be ce 
dimension: 3

539.
id: 4279304465
cd -> e
b -> de
a -> de
meet-irreducibles: c d ce abde acde bcde 
max antichain: abde acde bcde 
dimension: 3

540.
id: 4145348881
cd -> e
bc -> de
ab -> cde
b -> e
a -> de
meet-irreducibles: c d be ce ade bde acde bcde 
max antichain: c d be 
dimension: 3

541.
id: 4115988753
cd -> e
b -> e
a -> cde
meet-irreducibles: c d bce bde acde bcde 
max antichain: bce bde acde 
dimension: 3

542.
id: 4113891601
cd -> e
bc -> de
ab -> cde
b -> e
a -> ce
meet-irreducibles: c d be ace bde acde bcde 
max antichain: c d be 
dimension: 3

543.
id: 4144431111
d -> e
c -> de
ab -> cde
meet-irreducibles: a b ae be ade bde acde bcde 
max antichain: a b 
dimension: 2

544.
id: 4145088257
bc -> de
ac -> de
c -> e
ab -> cde
b -> d
a -> d
meet-irreducibles: ad bd ce ade bde acde bcde 
max antichain: ad bd ce 
dimension: 3

545.
id: 4145087761
cd -> e
bc -> de
ab -> cde
b -> d
a -> de
meet-irreducibles: c bd ce ade bde acde bcde 
max antichain: c bd ade 
dimension: 3

546.
id: 4113630481
cd -> e
bc -> de
ab -> cde
b -> d
a -> ce
meet-irreducibles: c bd ace bde acde bcde 
max antichain: c bd 
dimension: 2

547.
id: 4149542917
d -> e
c -> e
ab -> cde
a -> de
meet-irreducibles: b bce ade bde acde bcde 
max antichain: bce ade bde 
dimension: 3

548.
id: 4145479685
d -> e
bc -> de
ac -> de
c -> e
ab -> cde
a -> e
meet-irreducibles: b ae be ce ade bde acde bcde 
max antichain: b ae ce 
dimension: 3

549.
id: 4114022405
ad -> ce
d -> e
bc -> de
c -> e
ab -> cde
a -> e
meet-irreducibles: b ae be ace bde acde bcde 
max antichain: b ae 
dimension: 2

550.
id: 4145348629
d -> e
bc -> de
ab -> cde
a -> de
meet-irreducibles: b c be ce ade bde acde bcde 
max antichain: b c ade 
dimension: 3

551.
id: 4113891349
d -> e
bc -> de
ab -> cde
a -> ce
meet-irreducibles: b c be ace bde acde bcde 
max antichain: b c 
dimension: 2

552.
id: 4145090833
ab -> cde
b -> de
a -> de
meet-irreducibles: cd ce ade bde acde bcde 
max antichain: cd ce ade bde 
dimension: 4

553.
id: 4113633553
ab -> cde
b -> de
a -> ce
meet-irreducibles: cd ace bde acde bcde 
max antichain: cd ace bde 
dimension: 3

554.
id: 4111798545
bc -> de
b -> e
a -> cde
meet-irreducibles: cd be ce bde acde bcde 
max antichain: cd be ce 
dimension: 3

555.
id: 3579121937
b -> e
a -> bcde
meet-irreducibles: cd bce bde cde bcde 
max antichain: cd bce bde 
dimension: 3

556.
id: 4145348869
bd -> e
bc -> de
c -> e
ab -> cde
a -> de
meet-irreducibles: b d be ce ade bde acde bcde 
max antichain: b d ce 
dimension: 3

557.
id: 4115988741
bd -> e
c -> e
a -> cde
meet-irreducibles: b d bce bde acde bcde 
max antichain: bce bde acde 
dimension: 3

558.
id: 4113891589
bd -> e
bc -> de
c -> e
ab -> cde
a -> ce
meet-irreducibles: b d be ace bde acde bcde 
max antichain: b d ace 
dimension: 3

559.
id: 4144431365
bd -> e
c -> de
ab -> cde
a -> e
meet-irreducibles: b d ae be ade bde acde bcde 
max antichain: b d ae 
dimension: 3

560.
id: 4113891377
ad -> ce
d -> e
bc -> de
ab -> cde
b -> e
a -> c
meet-irreducibles: ac be ace bde acde bcde 
max antichain: ac be 
dimension: 2

561.
id: 4145087749
be -> d
bc -> de
c -> e
ab -> cde
a -> de
meet-irreducibles: b bd ce ade bde acde bcde 
max antichain: b ce ade 
dimension: 3

562.
id: 4113630469
be -> d
bc -> de
c -> e
ab -> cde
a -> ce
meet-irreducibles: b bd ace bde acde bcde 
max antichain: b ace 
dimension: 2

563.
id: 4144301317
c -> de
ab -> cde
a -> de
meet-irreducibles: bd be ade bde acde bcde 
max antichain: bd be ade 
dimension: 3

564.
id: 4111795461
bc -> de
c -> e
a -> cde
meet-irreducibles: bd be ce bde acde bcde 
max antichain: bd be ce 
dimension: 3

565.
id: 4144300805
bd -> e
c -> de
ab -> cde
a -> d
meet-irreducibles: b ad be ade bde acde bcde 
max antichain: b ad 
dimension: 2

566.
id: 4111794453
cd -> e
bd -> e
bc -> de
a -> cde
meet-irreducibles: b c d be ce bde acde bcde 
max antichain: b c d 
dimension: 3

567.
id: 3579117845
cd -> e
bd -> e
bc -> e
a -> bcde
meet-irreducibles: b c d bce bde cde bcde 
max antichain: b c d 
dimension: 3

568.
id: 4111533333
be -> d
cd -> e
bc -> de
a -> cde
meet-irreducibles: b c bd ce bde acde bcde 
max antichain: b c 
dimension: 2

569.
id: 3574927637
bd -> e
bc -> de
a -> bcde
meet-irreducibles: b cd be ce bde cde bcde 
max antichain: b cd ce 
dimension: 3

570.
id: 4283760641
d -> e
c -> e
b -> e
a -> de
meet-irreducibles: bce abde acde bcde 
max antichain: bce abde acde 
dimension: 3

571.
id: 4279697409
d -> e
bc -> de
ac -> de
c -> e
ab -> de
b -> e
a -> e
meet-irreducibles: ae be ce abde acde bcde 
max antichain: ae be ce 
dimension: 3

572.
id: 4149673985
d -> e
ac -> de
c -> e
ab -> cde
b -> e
a -> e
meet-irreducibles: ae bce ade bde acde bcde 
max antichain: ae bce bde 
dimension: 3

573.
id: 4279566353
d -> e
bc -> de
b -> e
a -> de
meet-irreducibles: c be ce abde acde bcde 
max antichain: abde acde bcde 
dimension: 3

574.
id: 4149542929
d -> e
ab -> cde
b -> e
a -> de
meet-irreducibles: c bce ade bde acde bcde 
max antichain: c ade bde 
dimension: 3

575.
id: 4145479697
d -> e
bc -> de
ac -> de
ab -> cde
b -> e
a -> e
meet-irreducibles: c ae be ce ade bde acde bcde 
max antichain: c ae be 
dimension: 3

576.
id: 4279566593
bc -> de
c -> e
b -> e
a -> de
meet-irreducibles: d be ce abde acde bcde 
max antichain: d be ce 
dimension: 3

577.
id: 4149543169
c -> e
ab -> cde
b -> e
a -> de
meet-irreducibles: d bce ade bde acde bcde 
max antichain: bce ade bde 
dimension: 3

578.
id: 4145479937
bc -> de
ac -> de
c -> e
ab -> cde
b -> e
a -> e
meet-irreducibles: d ae be ce ade bde acde bcde 
max antichain: d ae be ce 
dimension: 4

579.
id: 4114022657
ad -> ce
bc -> de
c -> e
ab -> cde
b -> e
a -> e
meet-irreducibles: d ae be ace bde acde bcde 
max antichain: d ae be 
dimension: 3

580.
id: 4145087745
bc -> de
c -> e
ab -> cde
b -> d
a -> de
meet-irreducibles: bd ce ade bde acde bcde 
max antichain: bd ce ade 
dimension: 3

581.
id: 4113630465
bc -> de
c -> e
ab -> cde
b -> d
a -> ce
meet-irreducibles: bd ace bde acde bcde 
max antichain: bd ace 
dimension: 2

582.
id: 4144300801
c -> de
ab -> cde
b -> e
a -> d
meet-irreducibles: ad be ade bde acde bcde 
max antichain: ad be 
dimension: 2

583.
id: 4145086737
cd -> e
ab -> cde
b -> de
a -> de
meet-irreducibles: c d ce ade bde acde bcde 
max antichain: c ade bde 
dimension: 3

584.
id: 4113629457
cd -> e
ab -> cde
b -> de
a -> ce
meet-irreducibles: c d ace bde acde bcde 
max antichain: c d 
dimension: 2

585.
id: 4111794449
cd -> e
bc -> de
b -> e
a -> cde
meet-irreducibles: c d be ce bde acde bcde 
max antichain: c d be 
dimension: 3

586.
id: 3579117841
cd -> e
b -> e
a -> bcde
meet-irreducibles: c d bce bde cde bcde 
max antichain: bce bde cde 
dimension: 3

587.
id: 4111533329
cd -> e
bc -> de
b -> d
a -> cde
meet-irreducibles: c bd ce bde acde bcde 
max antichain: c bd 
dimension: 2

588.
id: 4145348613
d -> e
bc -> de
c -> e
ab -> cde
a -> de
meet-irreducibles: b be ce ade bde acde bcde 
max antichain: b ce ade 
dimension: 3

589.
id: 4115988485
d -> e
c -> e
a -> cde
meet-irreducibles: b bce bde acde bcde 
max antichain: bce bde acde 
dimension: 3

590.
id: 4113891333
d -> e
bc -> de
c -> e
ab -> cde
a -> ce
meet-irreducibles: b be ace bde acde bcde 
max antichain: b ace 
dimension: 2

591.
id: 4144431109
d -> e
c -> de
ab -> cde
a -> e
meet-irreducibles: b ae be ade bde acde bcde 
max antichain: b ae 
dimension: 2

592.
id: 4111794197
d -> e
bc -> de
a -> cde
meet-irreducibles: b c be ce bde acde bcde 
max antichain: b c 
dimension: 2

593.
id: 4111536401
b -> de
a -> cde
meet-irreducibles: cd ce bde acde bcde 
max antichain: cd ce bde 
dimension: 3

594.
id: 3574927633
bc -> de
b -> e
a -> bcde
meet-irreducibles: cd be ce bde cde bcde 
max antichain: cd be ce 
dimension: 3

595.
id: 4144300293
bd -> e
c -> de
ab -> cde
a -> de
meet-irreducibles: b d be ade bde acde bcde 
max antichain: b d 
dimension: 2

596.
id: 4111794437
bd -> e
bc -> de
c -> e
a -> cde
meet-irreducibles: b d be ce bde acde bcde 
max antichain: b d ce 
dimension: 3

597.
id: 4111533317
be -> d
bc -> de
c -> e
a -> cde
meet-irreducibles: b bd ce bde acde bcde 
max antichain: b ce 
dimension: 2

598.
id: 4110746885
c -> de
a -> cde
meet-irreducibles: bd be bde acde bcde 
max antichain: bd be acde 
dimension: 3

599.
id: 3574923541
cd -> e
bd -> e
bc -> de
a -> bcde
meet-irreducibles: b c d be ce bde cde bcde 
max antichain: b c d 
dimension: 3

600.
id: 3574662421
be -> d
cd -> e
bc -> de
a -> bcde
meet-irreducibles: b c bd ce bde cde bcde 
max antichain: b c 
dimension: 2

601.
id: 4279566337
d -> e
bc -> de
c -> e
b -> e
a -> de
meet-irreducibles: be ce abde acde bcde 
max antichain: abde acde bcde 
dimension: 3

602.
id: 4149542913
d -> e
c -> e
ab -> cde
b -> e
a -> de
meet-irreducibles: bce ade bde acde bcde 
max antichain: bce ade bde 
dimension: 3

603.
id: 4145479681
d -> e
bc -> de
ac -> de
c -> e
ab -> cde
b -> e
a -> e
meet-irreducibles: ae be ce ade bde acde bcde 
max antichain: ae be ce 
dimension: 3

604.
id: 4114022401
ad -> ce
d -> e
bc -> de
c -> e
ab -> cde
b -> e
a -> e
meet-irreducibles: ae be ace bde acde bcde 
max antichain: ae be 
dimension: 2

605.
id: 4279304209
d -> e
b -> de
a -> de
meet-irreducibles: c ce abde acde bcde 
max antichain: abde acde bcde 
dimension: 3

606.
id: 4145348625
d -> e
bc -> de
ab -> cde
b -> e
a -> de
meet-irreducibles: c be ce ade bde acde bcde 
max antichain: c be ade 
dimension: 3

607.
id: 4113891345
d -> e
bc -> de
ab -> cde
b -> e
a -> ce
meet-irreducibles: c be ace bde acde bcde 
max antichain: c be 
dimension: 2

608.
id: 4279304449
c -> e
b -> de
a -> de
meet-irreducibles: d ce abde acde bcde 
max antichain: abde acde bcde 
dimension: 3

609.
id: 4145348865
bc -> de
c -> e
ab -> cde
b -> e
a -> de
meet-irreducibles: d be ce ade bde acde bcde 
max antichain: d be ce 
dimension: 3

610.
id: 4115988737
c -> e
b -> e
a -> cde
meet-irreducibles: d bce bde acde bcde 
max antichain: bce bde acde 
dimension: 3

611.
id: 4113891585
bc -> de
c -> e
ab -> cde
b -> e
a -> ce
meet-irreducibles: d be ace bde acde bcde 
max antichain: d be ace 
dimension: 3

612.
id: 4144431361
c -> de
ab -> cde
b -> e
a -> e
meet-irreducibles: d ae be ade bde acde bcde 
max antichain: d ae be 
dimension: 3

613.
id: 4111533313
bc -> de
c -> e
b -> d
a -> cde
meet-irreducibles: bd ce bde acde bcde 
max antichain: bd ce 
dimension: 2

614.
id: 4111532305
cd -> e
b -> de
a -> cde
meet-irreducibles: c d ce bde acde bcde 
max antichain: c d 
dimension: 2

615.
id: 3574923537
cd -> e
bc -> de
b -> e
a -> bcde
meet-irreducibles: c d be ce bde cde bcde 
max antichain: c d be 
dimension: 3

616.
id: 3574662417
cd -> e
bc -> de
b -> d
a -> bcde
meet-irreducibles: c bd ce bde cde bcde 
max antichain: c bd 
dimension: 2

617.
id: 4144300037
d -> e
c -> de
ab -> cde
a -> de
meet-irreducibles: b be ade bde acde bcde 
max antichain: b ade 
dimension: 2

618.
id: 4111794181
d -> e
bc -> de
c -> e
a -> cde
meet-irreducibles: b be ce bde acde bcde 
max antichain: b ce 
dimension: 2

619.
id: 3574923285
d -> e
bc -> de
a -> bcde
meet-irreducibles: b c be ce bde cde bcde 
max antichain: b c 
dimension: 2

620.
id: 4044427537
b -> cde
a -> cde
meet-irreducibles: cd ce de acde bcde 
max antichain: cd ce de 
dimension: 3

621.
id: 3574665489
b -> de
a -> bcde
meet-irreducibles: cd ce bde cde bcde 
max antichain: cd ce bde 
dimension: 3

622.
id: 4110745861
bd -> e
c -> de
a -> cde
meet-irreducibles: b d be bde acde bcde 
max antichain: b d 
dimension: 2

623.
id: 4279304193
d -> e
c -> e
b -> de
a -> de
meet-irreducibles: ce abde acde bcde 
max antichain: abde acde bcde 
dimension: 3

624.
id: 4145348609
d -> e
bc -> de
c -> e
ab -> cde
b -> e
a -> de
meet-irreducibles: be ce ade bde acde bcde 
max antichain: be ce ade 
dimension: 3

625.
id: 4115988481
d -> e
c -> e
b -> e
a -> cde
meet-irreducibles: bce bde acde bcde 
max antichain: bce bde acde 
dimension: 3

626.
id: 4113891329
d -> e
bc -> de
c -> e
ab -> cde
b -> e
a -> ce
meet-irreducibles: be ace bde acde bcde 
max antichain: be ace 
dimension: 2

627.
id: 4144431105
d -> e
c -> de
ab -> cde
b -> e
a -> e
meet-irreducibles: ae be ade bde acde bcde 
max antichain: ae be 
dimension: 2

628.
id: 4145086481
d -> e
ab -> cde
b -> de
a -> de
meet-irreducibles: c ce ade bde acde bcde 
max antichain: c ade bde 
dimension: 3

629.
id: 4111794193
d -> e
bc -> de
b -> e
a -> cde
meet-irreducibles: c be ce bde acde bcde 
max antichain: c be 
dimension: 2

630.
id: 4278255873
c -> de
b -> de
a -> de
meet-irreducibles: d e abde acde bcde 
max antichain: abde acde bcde 
dimension: 3

631.
id: 4145086721
c -> e
ab -> cde
b -> de
a -> de
meet-irreducibles: d ce ade bde acde bcde 
max antichain: ce ade bde 
dimension: 3

632.
id: 4113629441
c -> e
ab -> cde
b -> de
a -> ce
meet-irreducibles: d ace bde acde bcde 
max antichain: d ace 
dimension: 2

633.
id: 4144300289
c -> de
ab -> cde
b -> e
a -> de
meet-irreducibles: d be ade bde acde bcde 
max antichain: d be 
dimension: 2

634.
id: 4111794433
bc -> de
c -> e
b -> e
a -> cde
meet-irreducibles: d be ce bde acde bcde 
max antichain: d be ce 
dimension: 3

635.
id: 3579117825
c -> e
b -> e
a -> bcde
meet-irreducibles: d bce bde cde bcde 
max antichain: bce bde cde 
dimension: 3

636.
id: 3574662401
bc -> de
c -> e
b -> d
a -> bcde
meet-irreducibles: bd ce bde cde bcde 
max antichain: bd ce 
dimension: 2

637.
id: 4044423441
cd -> e
b -> cde
a -> cde
meet-irreducibles: c d ce de acde bcde 
max antichain: c d 
dimension: 2

638.
id: 3574661393
cd -> e
b -> de
a -> bcde
meet-irreducibles: c d ce bde cde bcde 
max antichain: c d 
dimension: 2

639.
id: 4110745605
d -> e
c -> de
a -> cde
meet-irreducibles: b be bde acde bcde 
max antichain: b acde 
dimension: 2

640.
id: 3507556625
b -> cde
a -> bcde
meet-irreducibles: cd ce de cde bcde 
max antichain: cd ce de 
dimension: 3

641.
id: 4278255617
d -> e
c -> de
b -> de
a -> de
meet-irreducibles: e abde acde bcde 
max antichain: abde acde bcde 
dimension: 3

642.
id: 4145086465
d -> e
c -> e
ab -> cde
b -> de
a -> de
meet-irreducibles: ce ade bde acde bcde 
max antichain: ce ade bde 
dimension: 3

643.
id: 4113629185
d -> e
c -> e
ab -> cde
b -> de
a -> ce
meet-irreducibles: ace bde acde bcde 
max antichain: ace bde 
dimension: 2

644.
id: 4144300033
d -> e
c -> de
ab -> cde
b -> e
a -> de
meet-irreducibles: be ade bde acde bcde 
max antichain: be ade 
dimension: 2

645.
id: 4111794177
d -> e
bc -> de
c -> e
b -> e
a -> cde
meet-irreducibles: be ce bde acde bcde 
max antichain: be ce 
dimension: 2

646.
id: 3579117569
d -> e
c -> e
b -> e
a -> bcde
meet-irreducibles: bce bde cde bcde 
max antichain: bce bde cde 
dimension: 3

647.
id: 4111532049
d -> e
b -> de
a -> cde
meet-irreducibles: c ce bde acde bcde 
max antichain: c bde 
dimension: 2

648.
id: 3574923281
d -> e
bc -> de
b -> e
a -> bcde
meet-irreducibles: c be ce bde cde bcde 
max antichain: c be 
dimension: 2

649.
id: 4144038145
c -> de
ab -> cde
b -> de
a -> de
meet-irreducibles: d e ade bde acde bcde 
max antichain: d e 
dimension: 2

650.
id: 4111532289
c -> e
b -> de
a -> cde
meet-irreducibles: d ce bde acde bcde 
max antichain: d ce 
dimension: 2

651.
id: 4110745857
c -> de
b -> e
a -> cde
meet-irreducibles: d be bde acde bcde 
max antichain: d be 
dimension: 2

652.
id: 3574923521
bc -> de
c -> e
b -> e
a -> bcde
meet-irreducibles: d be ce bde cde bcde 
max antichain: d be ce 
dimension: 3

653.
id: 3507552529
cd -> e
b -> cde
a -> bcde
meet-irreducibles: c d ce de cde bcde 
max antichain: c d 
dimension: 2

654.
id: 4144037889
d -> e
c -> de
ab -> cde
b -> de
a -> de
meet-irreducibles: e ade bde acde bcde 
max antichain: ade bde 
dimension: 2

655.
id: 4111532033
d -> e
c -> e
b -> de
a -> cde
meet-irreducibles: ce bde acde bcde 
max antichain: ce bde 
dimension: 2

656.
id: 4110745601
d -> e
c -> de
b -> e
a -> cde
meet-irreducibles: be bde acde bcde 
max antichain: be acde 
dimension: 2

657.
id: 3574923265
d -> e
bc -> de
c -> e
b -> e
a -> bcde
meet-irreducibles: be ce bde cde bcde 
max antichain: be ce 
dimension: 2

658.
id: 3574661137
d -> e
b -> de
a -> bcde
meet-irreducibles: c ce bde cde bcde 
max antichain: c bde 
dimension: 2

659.
id: 4110483713
c -> de
b -> de
a -> cde
meet-irreducibles: d e bde acde bcde 
max antichain: d e 
dimension: 2

660.
id: 4044423425
c -> e
b -> cde
a -> cde
meet-irreducibles: d ce de acde bcde 
max antichain: d ce 
dimension: 2

661.
id: 3574661377
c -> e
b -> de
a -> bcde
meet-irreducibles: d ce bde cde bcde 
max antichain: d ce 
dimension: 2

662.
id: 4110483457
d -> e
c -> de
b -> de
a -> cde
meet-irreducibles: e bde acde bcde 
max antichain: bde acde 
dimension: 2

663.
id: 4044423169
d -> e
c -> e
b -> cde
a -> cde
meet-irreducibles: ce de acde bcde 
max antichain: ce de 
dimension: 2

664.
id: 3574661121
d -> e
c -> e
b -> de
a -> bcde
meet-irreducibles: ce bde cde bcde 
max antichain: ce bde 
dimension: 2

665.
id: 4043374849
c -> de
b -> cde
a -> cde
meet-irreducibles: d e de acde bcde 
max antichain: d e 
dimension: 2

666.
id: 3573612801
c -> de
b -> de
a -> bcde
meet-irreducibles: d e bde cde bcde 
max antichain: d e 
dimension: 2

667.
id: 3507552513
c -> e
b -> cde
a -> bcde
meet-irreducibles: d ce de cde bcde 
max antichain: d ce 
dimension: 2

668.
id: 4043374593
d -> e
c -> de
b -> cde
a -> cde
meet-irreducibles: e de acde bcde 
max antichain: acde bcde 
dimension: 2

669.
id: 3573612545
d -> e
c -> de
b -> de
a -> bcde
meet-irreducibles: e bde cde bcde 
max antichain: bde cde 
dimension: 2

670.
id: 3507552257
d -> e
c -> e
b -> cde
a -> bcde
meet-irreducibles: ce de cde bcde 
max antichain: ce de 
dimension: 2

671.
id: 3506503937
c -> de
b -> cde
a -> bcde
meet-irreducibles: d e de cde bcde 
max antichain: d e 
dimension: 2

672.
id: 3506503681
d -> e
c -> de
b -> cde
a -> bcde
meet-irreducibles: e de cde bcde 
max antichain: e 
dimension: 1

\end{verbatim} 

\newpage

\section{}





\subsection{Representation of size 4 geometries} \label{rep size 4} 
\chead{\sc Appendix B.1: Representation of size 4 geometries} 
\input{anc/Appendix_B.1} 

\subsection{Representation of size 5 geometries} \label{rep size 5} 
\chead{\sc Appendix B.2: Representation of size 5 geometries} 
\input{anc/Appendix_B.2} 

\newpage
\subsection{Impossible Geometries} \label{impossibles} \chead{\sc Appendix B.3: Impossible Geometries} 
\input{anc/Appendix_B.3} 


\end{document}